\documentclass{amsart}
\usepackage{amssymb}
\usepackage{hyperref}
\usepackage{url}
\usepackage[all]{xy}

\def\yndef2{7.1}

\def\ksp{10}
\def\exgm{5.2}
\def\korankk{13.4}
\def\thirteenone{5.7}
\def\selmerdef{2.1}

\newtheorem{thm}[equation]{Theorem}
\newtheorem{lem}[equation]{Lemma}
\newtheorem{cor}[equation]{Corollary}
\newtheorem{prop}[equation]{Proposition}
\newtheorem{conj}[equation]{Conjecture}

\newtheorem*{Sconj}{Conjecture $\St(K/k,S,T,S')$}

\theoremstyle{definition}
\newtheorem{rem}[equation]{Remark}

\newtheorem{defn}[equation]{Definition}

\renewcommand{\theenumi}{(\roman{enumi})}

\numberwithin{equation}{section}

\newfont{\cyrr}{wncyr10}

\def\Z{\mathbb{Z}}
\def\Q{\mathbb{Q}}

\def\R{\mathbb{R}}
\def\C{\mathbb{C}}

\def\bN{\mathbf{N}}
\def\bm{\mathbf{m}}

\def\bu{\mathbf{u}}

\def\bw{\mathbf{w}}
\def\j{\mathbf{j}}

\def\Zp{\Z_p}

\def\Qp{\Q_p}

\def\A{\mathcal{A}}
\def\O{\mathcal{O}}
\def\P{\mathcal{P}}

\def\N{\mathcal{N}}

\def\cR{\mathcal{R}}
\def\cF{\mathcal{F}}
\def\cL{\mathcal{L}}

\def\DD{\mathcal{D}}
\def\BB{\mathcal{B}}

\def\m{\mathfrak{m}}
\def\n{\mathfrak{n}}
\def\p{\mathfrak{p}}
\def\q{\mathfrak{q}}
\def\d{\mathfrak{d}}
\def\fQ{\mathfrak{Q}}

\def\Hom{\mathrm{Hom}}
\def\Gal{\mathrm{Gal}}

\def\ord{\mathrm{ord}}

\def\unr{\mathrm{ur}}

\def\Frob{\mathrm{Fr}}

\def\St{\mathrm{St}}

\def\mlog{\mathrm{log}}
\def\Artin{\mathrm{Art}}

\def\too{\longrightarrow}

\def\isom{\xrightarrow{\sim}}
\def\isomm{\xrightarrow{\;\sim\;}}
\def\hookto{\hookrightarrow}
\def\onto{\twoheadrightarrow}

\def\bmu{\boldsymbol{\mu}}

\def\w#1{\wedge^{#1}}

\def\wk#1#2{#2_1 \wedge \cdots \wedge #2_{#1}}
\def\ZG{\Z[G]}
\def\ZH{\Z[H]}
\def\ZGamma{\Z[\Gamma]}

\def\abs#1#2{{|#1|_{#2}}}
\def\Nm{{\mathbf N}}
\def\ex{e_\chi}
\def\stick{\theta}
\def\one{\mathbf{1}}
\def\bvarphi{\boldsymbol{\varphi}}

\def\bpsi{\boldsymbol{\psi}}
\def\lat#1#2#3{{\w{#1,0}#3}}
\def\Tw{\mathrm{Tw}}

\def\si{\eta^{\mlog}}
\def\sa#1{\eta^\Artin_{#1}}

\def\mt#1{\overline{#1}}
\def\M{\mathcal{M}}

\def\stsys{\boldsymbol{\delta}}
\def\SS{\mathbf{SS}}
\def\KS{\mathbf{KS}}
\def\D{\O}

\def\Mchi{\M_\chi}
\def\Np{\N_p}
\def\stesys{\boldsymbol{\kappa}^\mathrm{St}}
\def\stek{\kappa^\mathrm{St}}

\title[Refined class number formulas for $\mathbb{G}_m$]
   {Refined class number formulas for $\mathbb{G}_m$}

\author{Barry Mazur}
\address{Department of Mathematics, 
Harvard University,
Cambridge, MA 02138, 
USA}
\email{\href{mailto:mazur@math.harvard.edu}{mazur@math.harvard.edu}}
\author{Karl Rubin}
\address{Department of Mathematics, 
UC Irvine,
Irvine, CA 92697, 
USA}
\email{\href{mailto:krubin@math.uci.edu}{krubin@math.uci.edu}}

\date{\today}

\subjclass[2010]{Primary 11R42, 11R27; Secondary 11R23, 11R29}

\thanks{This material is based upon work supported by the 
National Science Foundation under grants DMS-1302409 and DMS-1065904.}

\begin{document}

\begin{abstract}
We formulate a generalization of a ``refined class number formula'' of Darmon.  
Our conjecture deals with Stickelberger-type elements formed from 
generalized Stark units, and has two parts: the ``order of vanishing'' 
and the ``leading term''.  Using the theory of Kolyvagin systems 
we prove a large part of this conjecture when 
the order of vanishing of the corresponding complex $L$-function is $1$.
\end{abstract}

\maketitle

\tableofcontents

\section{Introduction}

In \cite{darmon}, Darmon conjectured a ``refined class number formula'' for real 
quadratic fields, inspired by work of Gross \cite{gross}, of 
the first author and Tate \cite{mazurtate}, and of Hayes \cite{hayes}.  
The common setting for these conjectures included a finite abelian extension $L/K$ and a 
Stickelberger-type element $\theta \in \Z[\Gal(L/K)]$.  
In analogy with the Birch and Swinnerton-Dyer conjecture, 
these conjectures predicted the ``order of vanishing'' (a nonnegative integer $r$ 
such that $\theta$ lies in the $r$-th power of the augmentation ideal 
$\A$ of $\Z[\Gal(L/K)]$) and the ``leading term'' (the image of $\theta$ in $\A^r/\A^{r+1}$) 
of $\theta$.

In \cite{darmon.conj}, we proved most (the ``non-$2$-part'') of Darmon's conjecture, 
using the theory of Kolyvagin systems \cite{kolysys}.  The key idea is that 
in nice situations, the space of Kolyvagin systems is a free $\Zp$-module of rank 
one, and hence two Kolyvagin systems that agree ``at $n = 1$'' must be equal.  
Darmon's conjecture for $n = 1$ follows from the classical evaluation of $L'(0,\chi)$ 
for a real quadratic Dirichlet character $\chi$.

In this paper we attempt to generalize both the statement and proof of 
Darmon's conjecture.  To generalize the statement we rely on a 
suitable version of Stark's conjectures.  Namely, given a finite abelian 
tower of number fields $L/K/k$, our proposed Conjecture \ref{conj} relates 
the so-called ``Rubin-Stark'' elements $\epsilon_{L,S_L}$ 
attached to $L/k$ (see \S\ref{RSconj}) 
with an ``algebraic regulator'' (see Definition \ref{5.3})
constructed from Rubin-Stark elements $\epsilon_{K,S_L}$ attached to $K/k$ and $L$.  
Similar generalizations of Darmon's conjecture have recently been proposed independently 
by Sano \cite[Conjecture 4]{sano} and Popescu \cite{popescu2}.

Our conjecture has two parts, the ``order of vanishing'' and the ``leading term''.
We prove a large portion of the order of vanishing part of the conjecture in Theorem \ref{ordvan}.
We prove a large part of the leading term statement in Theorem \ref{9.7} following the method of 
\cite{darmon.conj}, but only under the rather strong assumption that the order of vanishing 
(the ``core rank'', in the language of \cite{gen.kolysys}) is one.  
As $L$ varies, the elements $\epsilon_{L,S_L}$ form an Euler system, 
and the elements $\epsilon_{K,S_L}$ form what we call a Stark system.  
When the order of vanishing is one we can relate these systems and prove 
the leading term formula.  In the final section we prove a weakened 
version of the leading term statement for general $r$, under some additional 
hypotheses.

\subsection*{Notation}

Suppose throughout this paper that 
$\D$ is an integral domain with field of fractions $F$, and let $R = \D[\Gamma]$
with a finite abelian group $\Gamma$.  
We are mainly interested in the case where $\O = \Z$ or $\Zp$ for some prime $p$.

If $M$ is an $R$-module, we let $M^* := \Hom_R(M,R)$.  
If $\rho \in R$, then $M[\rho]$ will denote the kernel of multiplication by 
$\rho$ in $M$.

If $r \ge 0$, then $\wedge^r M$ (or $\wedge_R^r M$, if we need to emphasize the ring $R$) 
will denote the $r$-th exterior power of $M$ in the category of $R$-modules,
with the convention that $\wedge^0M = R$.
See Appendix \ref{wedges} for more on the exterior algebra that we use.  
In particular, in Definition \ref{mylat} we define an $R$-lattice 
$\lat{r}{}{M} \subset \wedge^rM \otimes F$, containing the image of 
$\wedge^rM$, that will play an important role.

\section{Unit groups}
\label{units}

Suppose $K/k$ is a finite abelian extension of number fields.  Let $\Gamma = \Gal(K/k)$
and $R = \ZGamma$. 
Fix a finite set $S$ of places of $k$ containing all infinite places and all
places ramified in $K/k$, and
a second finite set $T$ of places of $k$, disjoint from $S$.  Define:
\begin{gather*}
S_K = \{\text{places of $K$ lying above places in $S$}\}, \\
T_K = \{\text{places of $K$ lying above places in $T$}\}, \\
U_{K,S,T} = \{x \in K^\times : \text{$\abs{x}{w} = 1$ for all $w \notin S_K$, 
$x \equiv 1 \pmod{w}$ for all $w \in T_K$} \}.
\end{gather*}
We assume further that $K$ has no roots of unity congruent to $1$ modulo 
all places in $T_K$, so that $U_{K,S,T}$ is a free $\Z$-module.  
When there is no fear of confusion, we will suppress the $S$ and $T$ and write 
$U_K := U_{K,S,T}$

Suppose now that $L$ is a finite abelian extension of $k$ containing $K$.  
Let $G := \Gal(L/k)$ and $H := \Gal(L/K)$, so $G/H = \Gamma$.  
Let $\A_H \subset \ZH$ be the augmentation ideal, the ideal generated by $\{h-1 : h \in H\}$.

\begin{cor}
\label{Ipair}
For every $s \le r$ and every $\rho \in \Q[\Gamma]$, Proposition \ref{Rpair} 
gives a canonical pairing
$$
(\lat{r}{}{U_K})[\rho] \times \w{r-s}\Hom_\Gamma(U_K,\ZGamma \otimes_\Z \A_H/\A_H^2) 
   \too (\lat{s}{}{U_K})[\rho] \otimes_\Z \A_H^{r-s}/\A_H^{r-s+1}.
$$
\end{cor}

\begin{proof}
Apply Proposition \ref{Rpair} with $B: = \oplus_{i \ge 0} \A_H^i/\A_H^{i+1}$ and $n=1$. 
\end{proof}

\section{A Stark conjecture over $\Z$}
\label{RSconj}

In this section we recall the so-called Rubin-Stark conjecture over $\Z$ for arbitrary 
order of vanishing from \cite{stark}.  When the order of vanishing (the integer $r$ below) 
is one, this is essentially the ``classical'' Stark conjecture over $\Z$ 
(see for example \cite[\S{IV.2}]{tate} and \cite[Proposition 2.5]{stark}).

Keep the finite abelian extension $K/k$ of number fields from \S\ref{units}, 
with $\Gamma = \Gal(K/k)$, and the sets $S, T$ of places of $K$.  
We define the Stickelberger function attached to $K/k$ (and $S$ and $T$) 
to be the meromorphic $\C[\Gamma]$-valued function
$$
\stick_{K/k}(s) = \stick_{K/k,S,T}(s) = \prod_{\p \notin S}(1-\Frob_{\p}^{-1}\Nm \p^{-s})^{-1} 
   \prod_{\p \in T}(1-\Frob_{\p}^{-1}\Nm \p^{1-s})
$$
where $\Frob_{\p} \in\Gamma$ is the Frobenius of the (unramified) prime $\p$.  
If $\chi \in \hat{\Gamma} := \Hom(\Gamma,\C^\times)$, then applying $\chi$ to the 
Stickelberger function yields the (modified at $S$ and $T$) Artin $L$-function
$$
\chi(\stick_{K/k}(s)) = L_{S,T}(K/k;\bar\chi,s).
$$

\begin{defn}
\label{def3.1}
If $w$ is a place of $K$ we write
$K_w$ for the completion of $K$ at $w$ and 
$\abs{\;\;}{w} : K_w \to \R^+ \cup \{0\}$ for the absolute value 
normalized so that
$$
\abs{x}{w} = 
\begin{cases}
 \pm x \text{~(the usual absolute value)} & \text{if $K_w = \R$}, \\
    x\bar{x}  &  \text{if $K_w = \C$}, \\
    \Nm w^{-\ord_w(x)}  &   \text{if $K_w$ is nonarchimedean}
\end{cases}
$$
where $\Nm w$ is the cardinality of the residue field of the finite place $w$.
\end{defn}

\begin{defn}
\label{def3.2}
Suppose now that $S' \subset S$ is a subset such that every $v \in S'$
splits completely in $K/k$.  Let $r = |S'| \ge 0$.  
Let $S'_K$ denote the set of primes of $K$ above $S'$, and 
let $W_{K,S'}$ denote the free abelian group on $S'_K$, so  
$W_{K,S'}$ is a free $\ZGamma$-module of rank $r$.

Define a $\ZGamma$-homomorphism
$
\si_K : U_K \to W_{K,S'} \otimes \R
$
by
$$
\si_K(u) = \sum_{w \in S'_K} w \otimes \log|u|_w.
$$
If $L$ is an abelian extension of $K$ with Galois group $H := \Gal(L/K)$, 
and $\A_H \subset \Z[H]$ is the augmentation ideal, let 
$[\;\cdot\;,L_w/K_w] : K_w^\times \to H$ denote the local Artin symbol (this is 
independent of the choice of place of $L$ above $w$) and define a $\ZGamma$-homomorphism 
$
\sa{L/K} : U_K \to W_{K,S'} \otimes_\Z \A_H/\A_H^2
$
by 
$$
\sa{L/K}(u) := \sum_{w \in S'_K} w \otimes ([u,L_w/K_w]-1).
$$
\end{defn}

\begin{defn}
Let
$$
\cR^\infty = \cR_{K,S,T,S'}^\infty : \w{r}U_K \otimes_{\Gamma} \wedge^r W_{K,S'}^* \too \R[\Gamma]
$$
be the classical regulator map induced by $\si_K : \wedge^r U_K \to \wedge^r W_{K,S'} \otimes \R$ and 
the natural isomorphism $\wedge^r W_{K,S'} \otimes \wedge^r W_{K,S'}^* \to \Z[\Gamma]$.
\end{defn}

Concretely, the map $\cR^\infty$ is given as follows.  
If $w \in S'_K$, let $w^* \in W_{K,S'}^*$ be the map
$$
w^*\biggl(\sum_{z \in S_K'}a_z z\biggr) := \sum_{\gamma\in\Gamma} a_{\gamma w}\, \gamma.
$$
If $v_1, \ldots, v_r$ is an ordering of 
the places in $S'$, and for each $i$ we choose a place $w_i$ of $K$ above $v_i$, then 
$\wk{r}{w}$ is a $\ZGamma$-basis of $\wedge^r W_{K,S'}$, and 
$\wk{r}{w^*}$ is the dual basis of $\wedge^r W_{K,S'}^*$.  Then
$$
\cR^\infty((\wk{r}{u}) \otimes (\wk{r}{w^*})) = \det \bigl(\sum_{\gamma\in\Gamma}\log\abs{u_i^\gamma}{w_j}\gamma^{-1} \bigr).
$$

\begin{defn}
\label{4.5}
Write $\one$ for the trivial character of $\Gamma$.  For every $\chi \in \hat{\Gamma}$
there is an idempotent
$$
\ex = |\Gamma|^{-1}\sum_{\gamma \in\Gamma} \chi(\gamma) \gamma^{-1} \in \C[\Gamma],
$$
and we define a nonnegative integer $r(\chi) = r(\chi,S)$ by
\begin{equation}
\label{(3.5)}
r(\chi) = \ord_{s=0}L_S(s,\bar\chi) = \dim_\C \ex \C U_{K} = 
\begin{cases}
|\{v \in S : \chi(\Gamma_v) = 1\}| & \text{if $\chi \ne \one$} \\
|S|-1 & \text{if $\chi = \one$}
\end{cases}
\end{equation}
where $\Gamma_v$ is the decomposition group of $v$ in $\Gamma$
(see for example \cite[Proposition I.3.4]{tate}).
If $r \ge 0$ is such that $S$ contains $r$ places that split completely in $K/k$, and $|S| \ge r+1$,   
then $r(\chi) \ge r$ for every $\chi \in \hat\Gamma$, and we let
$$
\rho_{K,r} := \sum_{\chi \in \hat{\Gamma}, r(\chi) \neq r} e_\chi \in\Q[\Gamma].
$$
\end{defn}

The following is the ``Stark conjecture over $\Z$'' that we will use.

\begin{Sconj}[= Conjecture B$'$ of \cite{stark}]
Suppose that:
\begin{enumerate}
\item
$S$ is a finite set of places of $k$ containing all archimedean places 
and all places ramifying in $K/k$,
\item
$T$ is a finite set of places of $K$, disjoint from $S$, such that
$U_{K,S,T}$ contains no roots of unity,
\item
$S' \subsetneq S$ contains only places that split completely in $K$.
\end{enumerate}
Let $r = |S'|$.  Then there is a unique element 
$$
\epsilon_{K} = \epsilon_{K,S,T,S'} \in (\lat{r}{}{U_{K,S,T}})[\rho_{K,r}] \otimes_\Gamma \wedge^r W^*_{K,S'}
$$ 
such that 
$$
\cR^\infty(\epsilon_{K}) = \lim_{s \to 0}s^{-r}\stick_{K/k}(s).
$$
\end{Sconj}

By Conjecture $\St(K/k)$ we will mean the conjecture that $\St(K/k,S,T,S')$ 
holds for all choices of $S$, $T$, and $S'$ satisfying the hypotheses above.

Recall that $\wedge^rW^*_{K,S'}$ is free of rank one over $\ZGamma$.  
The uniqueness of $\epsilon_{K,S,T,S'}$ is automatic because $\cR^\infty$ is 
injective on $(\lat{r}{}{U_{K,S,T}})[\rho_{K,r}] \otimes_\Gamma \wedge^r W^*_{K,S'}$ 
(see for example \cite[Lemma 2.7]{stark}).

Conjecture $\St(K/k,S,T,S')$ is known to be true in the following cases:
\begin{itemize}
\item
$r = 0$ (in which case $\epsilon_{K} := \stick_{K/k}(0) \in \ZGamma$, 
which was proved independently by Deligne and Ribet, Cassou-Nogu\`es, and Barsky),
\item
$K/k$ is quadratic (\cite[Theorem 3.5]{stark}),
\item
$k = \Q$ (proved by Burns in \cite[Theorem A]{burns}),
\item
$S-S'$ contains a prime that splits completely in $K/k$ (\cite[Proposition 3.1]{stark}).
\end{itemize}

\begin{lem}
\label{xtra}
Suppose that $S-S'$ contains a place that splits completely in $K/k$, and 
$|S-S'| \ge 2$.  Then $\epsilon_K = 0$ satisfies Conjecture $\St(K/k,S,T,S')$.
\end{lem}

\begin{proof}
In this case $r(\chi,S) > r = |S'|$ for every $\chi \in \hat{\Gamma}$, so 
$\lim_{s \to 0}s^{-r}\stick_{K/k}(s) = 0$ and 
$\rho_{K,r} = 1$ in Definition \ref{4.5} .  The lemma follows.
\end{proof}

\section{The Artin regulator}
\label{rcnf}

Fix a finite abelian extension $L/k$ of number fields, and an intermediate 
field $K$, $k \subset K \subset L$.  Let 
$G := \Gal(K/k)$, $H := \Gal(K/F)$ and $\Gamma := \Gal(F/k) = G/H$.  
$$
\xymatrix@R=15pt@C=5pt{
&&& L \\
\\
&&K \ar@{-}_{H}[uur] \\
k \ar@{-}_{\Gamma}[urr] \ar@{-}@/^1pc/^{G}[uuurrr]
}
$$
Fix a finite set $S$ of places of $k$ containing all archimedean places and 
all primes ramifying in $L/k$.  Fix a second finite set of primes $T$ of $k$, 
disjoint from $S$, such that $U_L = U_{L,S,T}$ contains no roots of unity.  

Suppose that we have a filtration $S' \subset S'' \subsetneq S$, where 
every $v \in S''$ splits completely in $K/k$, and every $v \in S'$ splits 
completely in $L/k$.
Let $r = |S'|$ and $s = |S''|-|S'|$.

For the rest of this section, we keep $S, S', S''$ and $T$ fixed, and we suppress 
them from the notation when possible.

For every subset $\Sigma \subset S''$, 
let $W_{K,\Sigma}$ denote the free abelian group on the set of primes of $K$ 
above $\Sigma$, and similarly with $L$ in place of $K$.  
Then $W_{K,\Sigma}$ is a free $\ZGamma$-module of rank $|\Sigma|$, 
we have 
\begin{equation} 
\label{sss}
W_{K,S''} = W_{K,S'} \oplus W_{K,S''-S'}, \quad 
   \wedge^{r+s} W_{K,S''} = \wedge^r W_{K,S'} \otimes_\Gamma \wedge^{s} W_{K,S''-S'},
\end{equation}
and the natural map $S_L \to S_K$, that takes a place of $L$ to its restriction to $K$, induces 
an isomorphism of free modules
\begin{equation} 
\label{sss'}
W_{L,S'} \otimes_G \ZGamma \isomm W_{K,S'}.
\end{equation}

Let $\sa{L/K} \in \Hom_\Gamma(U_K, W_{K,S''-S'} \otimes_\Z \A_H/\A_H^2)$ be the map of 
Definition \ref{def3.2}, with the augmentation ideal $\A_H$ as in \S\ref{units}.
Composition with $\sa{L/K}$ gives a $\ZGamma$-homo\-mor\-phism
\begin{equation}
\label{5.2}
W^*_{K,S''-S'} \too \Hom_\Gamma(U_K,\ZGamma \otimes_\Z \A_H/\A_H^2).
\end{equation}
Corollary \ref{Ipair} gives a canonical pairing
$$
(\lat{r+s}{}{U_K}) \times \w{s}\Hom_\Gamma(U_K,\ZGamma \otimes_\Z \A_H/\A_H^2) 
   \too (\lat{r}{}{U_K}) \otimes_\Z \A_H^{s}/\A_H^{s+1},
$$
and using \eqref{5.2} we can pull this back to a pairing
\begin{equation}
\label{labelthis}
(\lat{r+s}{}{U_K}) \times \w{s}W^*_{K,S''-S'} 
   \too (\lat{r}{}{U_K}) \otimes_\Z \A_H^{s}/\A_H^{s+1}.
\end{equation}

\begin{defn}
\label{5.3}
Tensoring both sides of \eqref{labelthis} with $\wedge^r W^*_{K,S'}$ and using \eqref{sss}, 
we define an algebraic regulator map $\cR_{L/K}^\Artin = \cR_{L/K,S,S',S''}^\Artin$
$$
\cR_{L/K}^\Artin : (\lat{r+s}{}{U_K}) \otimes_\Gamma \wedge^{r+s} W^*_{K,S''} 
   \too (\lat{r}{}{U_K}) \otimes_\Gamma \wedge^r W^*_{K,S'} \otimes_\Z \A_H^{s}/\A_H^{s+1}.
$$
\end{defn}

\begin{defn}
Let $\iota_{L/K} : \ZGamma \hookto \ZG$ denote the natural $\ZG$-module homomorphism 
that sends $\gamma \in \Gamma$ to $\sum_{g \in\gamma}g$, viewing $\gamma$ as an $H$-coset. 
Then $\iota_{L/K}$ is not a ring homomorphism, but rather 
\begin{equation}
\label{nothom}
\iota_{L/K}(\alpha)\iota_{L/K}(\beta) = [L:K]\iota_{L/K}(\alpha\beta).
\end{equation}
Note that $\iota_{L/K}$ is a $\ZG$-module isomorphism $\ZGamma \isom \ZG^H$.

As in \S\ref{units}, let 
$$
U_K^* := \Hom_\Gamma(U_K,\ZGamma), \quad U_L^* := \Hom_G(U_L,\ZG).
$$  
If $\varphi \in U_L^*$, then $\varphi(U_K) \subset \ZG^H$, and we define 
$\varphi^K = \iota_{L/K}^{-1} \circ \varphi|_{U_K} \in U_K^*$.

Let $j_{L/K} : U_K \hookto U_L$ denote the natural inclusion, and 
$\w{s}j_{L/K} : \w{s}U_K \to \w{s}U_L$ the induced map (if $s = 0$, we let 
$\w{0}j_{L/K} = \iota_{L/K} : \ZGamma \to \ZG$).
\end{defn}

\begin{lem}
\label{5.5}
$\w{r}j_{L/K}(\lat{r}{}{U_K}) \subset [L:K]^{\max\{0,r-1\}}\lat{r}{}{U_L}$.
\end{lem}

\begin{proof}
If $r = 0$ there is nothing to prove, so assume $r \ge 1$.  
Suppose $\varphi_1,\ldots,\varphi_r \in U_L^*$.  Let 
$\bvarphi = \wk{r}{\varphi} \in \w{r}U_L^*$, and 
$\bvarphi^K = \wk{r}{\varphi^K} \in \w{r}U_K^*$.  
Using \eqref{nothom} and the evaluation \eqref{detform} of $\bvarphi$ and $\bvarphi^K$ as 
determinants, we have a commutative diagram
$$
\xymatrix{
\lat{r}{}{U_L} \ar@{^(->}[r] & {\w{r} U_L}\otimes\Q \ar^-{\bvarphi}[r] & \Q[G] \\
\lat{r}{}{U_K} \ar@{^(->}[r] & {\w{r} U_K}\otimes\Q \ar^-{\bvarphi^K}[r]\ar^{\w{r}j_{L/K}}[u] & \Q[\Gamma] \ar_{[L:K]^{r-1}\iota_{L/K}}[u]
}
$$
By definition $\bvarphi^K(\lat{r}{}{U_K}) \subset \ZGamma$, 
so $\bvarphi(\w{r}j_{L/K}(\lat{r}{}{U_K})) \subset [L:K]^{r-1}\ZG$.  Since these $\bvarphi$ 
generate $\w{r}U_L^*$, this proves the lemma.
\end{proof}

\begin{lem}
\label{5.6}
The map $[L:K]^{-\max\{0,r-1\}}\w{r}j_{L/K} : \lat{r}{}{U_K} \to \lat{r}{}{U_L}$ 
and the inverse of the isomorphism \eqref{sss'}
induce a map
$$
\j_{L/K} : (\lat{r}{}{U_K}) \otimes_\Gamma \wedge^r W^*_{K,S'} 
   \too (\lat{r}{}{U_L}) \otimes_G \wedge^r W^*_{L,S'}.
$$
\end{lem}

\begin{proof}
Using \eqref{sss'} for the second equality, we have
\begin{align*}
(\lat{r}{}{U_K}) \otimes_\Gamma \wedge^r W^*_{K,S'} &= (\lat{r}{}{U_K}) \otimes_G \wedge^r W^*_{K,S'} \\
   &= (\lat{r}{}{U_K}) \otimes_G (\wedge^r W^*_{L,S'} \otimes_G \ZGamma) \\
   &= ((\lat{r}{}{U_K}) \otimes_G \ZGamma) \otimes_G \wedge^r W^*_{L,S'} \\
   &= (\lat{r}{}{U_K}) \otimes_G \wedge^r W^*_{L,S'}.
\end{align*}
Now the lemma follows from Lemma \ref{5.5}.
\end{proof}

\section{The conjecture}
\label{statement}

Let $L/K/k$, $G$, $H$, $\Gamma$, $S$, $T$, $S'$, $S''$ be as in \S\ref{rcnf}.
The hypotheses 
of Conjectures $\St(L/k,S,T,S')$ and $\St(K/k,S,T,S'')$ are all satisfied, and 
if those conjectures both hold they provide us with elements
\begin{align*}
\epsilon_L &:= \epsilon_{L,S,T,S'} \in (\lat{r}{}{U_L})[\rho_{L,r}] \otimes_G \wedge^r W^*_{L,S'} 
   \subset \w{r}{U_L} \otimes_G \wedge^r W^*_{L,S'} \otimes_\Z \Q, \\
   \epsilon_K &:= \epsilon_{K,S,T,S''} \in 
   (\lat{r+s}{}{U_K})[\rho_{K,r+s}] \otimes_\Gamma \wedge^{r+s} W^*_{K,S''} \\
      &\hskip2.5in\subset \w{r+s}{U_K} \otimes_\Gamma \wedge^{r+s} W^*_{K,S''} \otimes_\Z \Q.
\end{align*}

\begin{defn}
If $M$ is a $\ZH$-module, define the {\em twisted trace}
$$
\Tw_{L/K} : M \too M \otimes_\Z \ZH
$$
by $\Tw_{L/K}(m) := \sum_{h \in H} m^{h} \otimes h^{-1} \in M \otimes_\Z \ZH$.
\end{defn}

We will think of $\Tw_{L/K}(\epsilon_L)$ as a generalized Stickelberger element.
The following conjecture is inspired by conjectures in 
\cite{mazurtate, gross, grossletter, darmon}.

\begin{conj}
\label{conj}
With $(L/K/k,S,T,S',S'')$ as in \S\ref{rcnf}, 
suppose that Conjectures $\St(L/k,S,T,S')$ and $\St(K/k,S,T,S'')$ both hold.
\begin{enumerate}
\item
``Order of vanishing'':
$$
\Tw_{L/K}(\epsilon_{L}) \in (\lat{r}{}{U_L}) \otimes_G \wedge^r W^*_{L,S'} \otimes_\Z \A_H^{s}.
$$  
\item
``Leading term'': with the maps $\j_{L/K}$ of Lemma \ref{5.6} and $\cR_{L/K}^\Artin$ 
of Definition \ref{5.3}, we have
$$
\Tw_{L/K}(\epsilon_{L}) \equiv (\j_{L/K}\otimes 1)(\cR_{L/K}^\Artin(\epsilon_K)) 
$$
in $(\lat{r}{}{U_L}) \otimes_G \wedge^r W^*_{L,S'} \otimes_\Z \A_H^{s}/\A_H^{s+1}$.
\end{enumerate}
\end{conj}

\begin{rem}
Suppose that $k = \Q$, $K$ is a real quadratic field, and $L = K(\bmu_n)^+$ 
(the real subfield of the extension of $K$ generated by the $n$-th roots of unity) 
with $n$ prime to the conductor of $K/\Q$.  
Let $S' := \{\infty\}$ and $S'' := \{\infty\} \cup \{\ell : \ell \mid n\}$ (so $r=1$).
In this case 
$\St(L/k,S,T,S')$ and $\St(K/k,S,T,S'')$ are known to hold, 
and Conjecture \ref{conj} is essentially the same as Darmon's conjecture in 
\cite[\S4]{darmon}.  This case was studied in detail in \cite{darmon.conj}.

See \S\ref{r=1} for more about the case $r=1$.  
\end{rem}

\begin{prop}
\label{grossequiv}
If $r=0$ then Conjecture \ref{conj} is equivalent to the conjecture 
of Gross in \cite[Conjecture 7.6]{gross} and \cite{grossletter} 
(see Conjecture $\tilde{A}_\Z(L/K/k,S,T,s)$ of \cite{popescu}).
\end{prop}

Before proving Proposition \ref{grossequiv}, we have the following two lemmas.
Let $J_H := \ZG \A_H$ be the kernel of the natural projection $\ZG \onto \ZGamma$.

\begin{lem}
\label{auglem}
There are natural isomorphisms
\begin{enumerate}
\item
$H \isom \A_H/\A_H^2$, given by $h \mapsto (h-1)$,
\item
$\ZGamma \otimes_\Z \A_H^r/\A_H^{r+1} \isom J_H^r/J_H^{r+1}$ for every $r \ge 0$, 
given by $\gamma \otimes \alpha \mapsto \alpha\bar\gamma$, 
where $\bar\gamma$ is any lift of $\gamma$ to $\Z[G]$.
\end{enumerate}
\end{lem}

\begin{proof}
This is a standard exercise.
\end{proof}

\begin{lem}
\label{pregrossequiv}
Define $\psi : \ZG \to \ZG \otimes_\Z \ZH$ by 
$\psi(\rho) = \sum_{h \in H} h\rho \otimes h^{-1}$.  
Then:
\begin{enumerate}
\item
$\psi$ is an injective $\ZG$-module homomorphism (with $G$ acting on the left 
on $\ZG \otimes_\Z \ZH$),
\item
$\psi(h\rho) = \psi(\rho)h$ for every $h \in H$,
\item
$\psi(J_H^t)\subset \ZG \otimes_\Z \A_H^t$ for every $t \ge 0$,
\item
for every $t \ge 0$ there is a commutative diagram
$$
\xymatrix@C=50pt{
J_H^t/J_H^{t+1} \ar@{^(->}^-{\psi}[dr] \\
\ZGamma \otimes_\Z \A_H^t/\A_H^{t+1} \ar@{^(->}^{\iota_{L/K} \otimes 1}[r] 
   \ar^{\cong}[u] & \ZG \otimes_\Z \A_H^t/\A_H^{t+1}
}
$$
where the vertical map is the isomorphism of Lemma \ref{auglem}(ii).
\end{enumerate}
\end{lem}

\begin{proof}
The first two assertions are clear, and (iii) follows from (ii).

To check the commutativity of the diagram in (iv), take 
$\gamma \in \Gamma$ and $\alpha \in \A_H^t$.  Using (ii), the image of $\gamma \otimes \alpha$  
in $\ZG \otimes_\Z \A_H^t/\A_H^{t+1}$ by the upper path is 
$$
\psi(\alpha\bar\gamma) = 
   \sum_{h \in H} h \alpha\bar\gamma \otimes h^{-1} 
   = \sum_{h \in H}\bar\gamma h \otimes \alpha h^{-1},
$$ 
where $\bar\gamma$ is any lift of $\gamma$ to $G$.
The image of $\gamma \otimes \alpha$ by the lower path is $\sum_{h \in H}\bar\gamma h \otimes \alpha$.  
Since $\alpha(h^{-1}-1) \in \A_H^{t+1}$ for every $h$, these are equal in 
$\ZG \otimes_\Z \A_H^t/\A_H^{t+1}$.  
This shows that the diagram in (iv) commutes, and 
the injectivity of the map induced by $\psi$ now follows from the injectivity of $\iota_{L/K}$.
\end{proof}

\begin{proof}[Proof of Proposition \ref{grossequiv}]
Let $\psi$ be as in Lemma \ref{pregrossequiv}.
If $r=0$, then $\epsilon_L = \stick_{L/k}(0)$, so $\Tw_{L/k}(\epsilon_L) = \psi(\stick_{L/k}(0))$.  
Thus by Lemma \ref{pregrossequiv}(iii,iv), 
Conjecture \ref{conj} is equivalent in this case to the assertions
\begin{enumerate}
\item
$\stick_{L/k}(0) \in J_H^s$,
\item
$\stick_{L/k}(0) \equiv \cR_{L/K}^\Artin(\epsilon_K) \pmod{J_H^{s+1}}$, 
\end{enumerate}
where we view 
$\cR_{L/K}^\Artin(\epsilon_K) \in J_H^{s+1}$ via the isomorphism of Lemma \ref{auglem}(ii).  
This is the usual statement of Gross' conjecture 
\cite[Conjecture $\tilde{A}_\Z(L/K/k,S,T,s)$]{popescu}.
\end{proof}

\begin{prop}
\label{5.15}
If $s=0$ (i.e., if $S'' = S'$), then Conjecture \ref{conj} is true.
\end{prop}

\begin{proof}
Conjecture \ref{conj}(i) is vacuous when $r=0$, since by definition 
$$
\Tw_{L/K}(\epsilon_{L}) \in (\lat{r}{}{U_L}) \otimes_\Gamma \wedge^r W^*_{S'} \otimes_\Z \ZH 
   = (\lat{r}{}{U_L}) \otimes_\Gamma \wedge^r W^*_{S'} \otimes_\Z \A_H^0.
$$
Let $\bN_H = \sum_{h \in H}h$.
In $(\lat{r}{}{U_L}) \otimes_\Gamma \wedge^r W^*_{S'} \otimes \ZH/\A_H$ we have
\begin{equation}
\label{(5.8)}
\Tw_{L/K}(\epsilon_{L}) = \sum_{h \in H} \epsilon_L^{h} \otimes h 
   \equiv \sum_{h \in H} \epsilon_L^{h} \otimes 1
   = (\bN_H\epsilon_L) \otimes 1.
\end{equation}
If $r=0$, then since the image of the Stickelberger element 
$\stick_{L/k}(0)$ under the restriction map $\ZG \onto \ZGamma$ is $\stick_{K/k}(0)$, 
we have
$$
\bN_H\epsilon_L = \bN_H\stick_{L/k}(0) = \iota_{L/K}\stick_{K/k}(0) 
   = \j_{L/K}(\epsilon_K).
$$
By \eqref{(5.8)}, this proves Conjecture \ref{conj}(ii) when $r=0$.

Suppose now that $r > 0$.  Fix generators $\bw_L = \wk{r}{w}$ and $\bw^*_L = \wk{r}{w^*}$ 
of $\wedge^r W_{L,S'}$ and $\wedge^r W^*_{L,S'}$, respectively.   
Let $\bw_K$ and $\bw_K^*$ be the corresponding generators of $\wedge^r W_{K,S'}$ 
and $\wedge^r W^*_{K,S'}$ obtained by restricting the $w_i$ to $K$ 
(note that since $s=0$, we have $S' = S''$).
Choose $\bu_L \in \lat{r}{}{U_L}$ and $\bu_K \in \lat{r}{}{U_K}$ 
such that $\epsilon_L = \bu_L \otimes \bw^*_L$ 
and $\epsilon_K = \bu_K \otimes \bw^*_K$.  
Then \cite[Proposition 6.1]{stark} shows that 
$(\bN_H)^r\bu_L = (\w{r}j)(\bu_K)$, and so we also have
\begin{multline}
\label{(5.9)}
\bN_H\epsilon_L = [L:K]^{1-r}(\bN_H)^r\epsilon_L 
   = ([L:K]^{1-r}(\bN_H)^r\bu_L) \otimes \bw^*_L \\
   = ([L:K]^{1-r}(\w{r}j)(\bu_K)) \otimes \bw^*_L
   = \j_{L/K}(\epsilon_K).
\end{multline}
Since $s=0$, the map \eqref{labelthis} is just the map 
$\lat{r}{}{U_K} \to (\lat{r}{}{U_K}) \otimes_\Z \ZH/\A_H$ that sends 
$\bu$ to $\bu \otimes 1$, so $\cR_{L/K}^\Artin$ is the map 
$$
(\lat{r}{}{U_K}) \otimes_\Gamma \wedge^r W^*_{K,S'} 
   \too (\lat{r}{}{U_K}) \otimes_\Gamma \wedge^r W^*_{K,S'} \otimes \ZH/\A_H
$$ 
that sends $\bu \otimes \bw$ to $\bu \otimes \bw \otimes 1$.  Hence by \eqref{(5.8)} 
and \eqref{(5.9)} we have
$$
\Tw_{L/K}(\epsilon_{L}) \equiv \j_{L/K}(\epsilon_K) \otimes 1 
   = (\j_{L/K}\otimes 1)(\cR_{L/K}^\Artin(\epsilon_K)) 
$$
in $(\lat{r}{}{U_K}) \otimes_\Gamma \wedge^r W^*_{K,S'} \otimes_\Z \ZH/\A_H$, 
which is Conjecture \ref{conj}(ii).
\end{proof}

\begin{prop}
\label{L=K}
If $L = K$, then Conjecture \ref{conj} is true.
\end{prop}

\begin{proof}
If $S'' = S'$, then this follows from Proposition \ref{5.15}.  
If $S'' \ne S'$ then for every character $\chi$ of $\Gamma$, 
we have $r(\chi,S) \ge |S''| > |S'| = r$, so $\rho_{L,r} = 1$ in 
Definition \ref{4.5} and by definition $\epsilon_L = \epsilon_{L,S,T,S'} = 0$.
Further, we have $\A_H = 0$ in this case, so $\cR^\Artin_{K/K} = 0$ and 
Conjecture \ref{conj} holds.
\end{proof}

\section{Order of vanishing}
\label{ov}

Fix a number field $k$, and a set $S'$ of archimedean places of $k$.  
Let $r := |S'|$.  Let $T$ be a finite set of primes of $k$, containing at 
least one prime not dividing $2$, and containing primes of at least two different 
residue characteristics if $S'$ contains no real places.  (This ensures that 
an extension of $k$ in which all places in $S'$ split completely 
has no roots of unity congruent to one modulo all primes in $T$.)

For example (perhaps the most interesting example), 
$k$ could be a totally real field and $S'$ the set of all 
archimedean places, in which case $r = [k:\Q]$.

Fix a finite abelian extension $K$ of $k$ such that all places in $S'$ 
split completely in $K/k$, and all places in $T$ are unramified in $K/k$.  
Fix a finite set $S$ of places of $K$ disjoint from $T$, containing all 
archimedean places, all primes ramifying in $K/k$, and at least one 
place not in $S'$.  Let $\P$ be the set of all primes of $k$ not in $S \cup T$ that 
split completely in $K/k$, 
and let $\N$ be the set of all squarefree products of primes in $\P$.  

For every $\q\in\P$ suppose that $K(\q)$ is a finite abelian extension of $k$ containing $K$, 
such that $K(\q)/K$ is totally tamely ramified above $\q$ and unramified everywhere else, 
and all places above $S'$ split completely in $K(\q)/K$.
(For example, if $K$ contains the Hilbert class field of $k$ then we could take 
$K(\q)$ to be the compositum of $K$ with the ray class field of $k$ modulo $\q$.)
If $\n\in\N$ define $K(\n)$ to be the compositum of the fields $K(\q)$ for $\q$ 
dividing $\n$.  Ramification considerations show that 
all the $K(\q)$ are linearly disjoint over $K$, so 
if we define $H(\n) := \Gal(K(\n)/K)$ then
$$
H(\n) = \prod_{\q\mid\n} H(\q)
$$
and if $\m\mid\n$ we can view $H(\m)$ both as a quotient 
and a subgroup of $H(\n)$.
Let 
$$
\pi_\m : H(\n) \onto H(\m) \hookto H(\n),
$$
denote the projection map.

Let $S(\n) := S \cup \{\q : \q\mid\n\}$ and $S'(\n) := S' \cup \{\q : \q\mid\n\}$.  
Assume for the rest of this section that the generalized Stark conjecture 
$\St(K(\n)/k,S(\n),T,S')$ holds for every $\n\in\N$, with an element 
$$
\epsilon_{\n} := \epsilon_{K(\n),S(\n),T,S'} 
   \in \lat{r}{}{U_{K(\n),S(\n)}} \otimes \wedge^{r}W^*_{K(\n),S'}.
$$

\begin{lem}
\label{esrel}
If $\d\mid\n$ then 
$$
\sum_{\gamma\in H(\n/\d)} \gamma\epsilon_{\n} 
   = \left(\textstyle\prod_{\q\mid(\n/\d)}(1-\Frob_\q^{-1})\right)\j_{K(\n)/K(\d)}(\epsilon_{\d}).
$$
\end{lem}

\begin{proof}
This follows from \cite[Proposition 6.1]{stark} and the definition (Lemma \ref{5.6}) 
of $\j_{K(\n)/K(\d)}$.
\end{proof}

Let $\nu(\n)$ denote the number of prime factors of $\n$.

\begin{lem}
\label{ovlem}
We have
\begin{multline*}
\sum_{\gamma\in H(\n)}
   \gamma\epsilon_{\n} \otimes {\textstyle\prod_{\q\mid\n}(\pi_\q(\gamma)-1)} \\
   = \sum_{\d\mid\n} \sum_{\gamma\in H(\d)}\gamma\,\j_{K(\n)/K(\d)}(\epsilon_{\d}) 
   \otimes\textstyle \gamma \prod_{\q\mid(\n/\d)}(\pi_\d(\Frob_\q) - 1)
\end{multline*}
in $\lat{r}{}{U_{K(\n),S(\n)}} \otimes \wedge^{r}W^*_{K(\n),S'} \otimes \Z[H(\n)]$.
\end{lem}

\begin{proof}
Expanding gives
\begin{equation*}
\sum_{\gamma\in H(\n)}
   \gamma\epsilon_{\n} \otimes {\textstyle\prod_{\q\mid\n}(\pi_\q(\gamma)-1)}
      = \sum_{\gamma\in H(\n)}\sum_{\d\mid\n} (-1)^{\nu(\n/\d)} \gamma\epsilon_{\n}
      \otimes \pi_\d(\gamma).
\end{equation*}
For every $\d$ dividing $\n$, using Lemma \ref{esrel} we have
\begin{align*}
\sum_{\gamma\in H(\n)} \gamma\epsilon_{\n} \otimes \pi_\d(\gamma) 
   &= \sum_{\gamma\in H(\d)} \bigl(\gamma\textstyle\sum_{h \in H(\n/\d)} h\epsilon_{\n}\bigr)
      \otimes \gamma \\
   &= \sum_{\gamma\in H(\d)} 
      \gamma{\textstyle\prod_{\q\mid(\n/\d)}(1-\Frob_\q^{-1})\,\j_{K(\n)/K(\d)}(\epsilon_{\d})}
      \otimes \gamma \\
   &= \sum_{\gamma\in H(\d)}\gamma\,\j_{K(\n)/K(\d)}(\epsilon_{\d}) 
   \otimes \gamma\textstyle\prod_{\q\mid(\n/\d)}(1-\pi_\d(\Frob_\q)).
\end{align*}
Combining these identities proves the lemma.
\end{proof}

\begin{thm}
\label{ordvan}
Suppose that the Stark conjecture $\St(K(\n)/k,S(\n),T,S')$ 
holds for every $\n\in\N$.  Then for every $\n\in\N$, we have
$$
\Tw_{K(\n)/K}(\epsilon_{\n}) \in \lat{r}{}{U_{K(\n),S(\n)}} 
   \otimes \wedge^{r}W^*_{K(\n),S'} \otimes \A_{H(\n)}^{\nu(\n)}.
$$
In other words, Conjecture \ref{conj}(i) holds for $(K(\n)/K/k,S(\n),T,S',S'(\n))$.
\end{thm}

\begin{proof}
The proof, by induction on $\nu(\n)$, is essentially the same as that 
of \cite[Lemma 8.1]{darmon}.
In the equality of Lemma \ref{ovlem}, every term except possibly 
$\Tw_{K(\n)/K}(\epsilon_{\n})$ (the summand on the right with $\d=\n$) 
lies in 
$\lat{r}{}{U_{K(\n),S(\n)}} \otimes\wedge^{r}W^*_{K(\n),S'} \otimes \A_{H(\n)}^{\nu(\n)}$ 
by our induction hypothesis.  Therefore $\Tw_{K(\n)/K}(\epsilon_{\n})$ 
does as well.
\end{proof}

\section{The case $K = k$}
\label{Kk}

In this section we consider the case $K = k$. 
Let $S'$, $S$, $T$, $\N$, $k(\q)$, $k(\n)$, $H(\n)$, $S(\n)$, $S'(\n)$ 
be as in \S\ref{ov}, and recall that $r := |S'|$.  We will show under mild hypotheses 
that Conjecture \ref{conj} holds in this case (with both sides of Conjecture \ref{conj}(ii) 
equal to zero).  This is needed for the proof of Theorem \ref{9.7} below, because 
our general techniques only work for nontrivial characters of $K/k$.

\begin{lem}
\label{lem9.6}
Suppose that $S'$ does not contain all archimedean places of $k$.  Then 
$\epsilon_{k(\n),S(\n),T,S'} = 0$ for every $\n \ne 1$.
\end{lem}

\begin{proof}
Let $w$ be an archimedean place not in $S'$.  By definition $k(\n)/k$ is unramified 
outside of $\n$, so $w$ splits completely in $k(\n)/k$.  Hence if $\n\ne 1$ 
then $\epsilon_{k(\n),S(\n),T,S'} = 0$ by Lemma \ref{xtra}.
\end{proof}

\begin{thm}
\label{K=kthm}
Suppose $\n\in\N$ and Conjecture $\St(k(\n)/k)$ holds.
If $|S - S'| \ge 2$, or if $S'$ does not contain all archimedean places of $k$, 
then Conjecture \ref{conj} holds for $(k(\n)/k/k,S(\n),T,S',S'(\n))$.
\end{thm}

\begin{proof}
Conjecture \ref{conj}(i) holds by Theorem \ref{ordvan}, 
and Conjecture \ref{conj}(ii) holds when $\n = 1$ by Proposition \ref{L=K}.
To prove the theorem we will show that for every $\n \ne 1$,
\begin{align}
\label{(7.3)}
\Tw_{k(\n)/k}(\epsilon_{k(\n),S(\n),T,S'}) 
   &\in \lat{r}{}{U_{k(\n),S_0(\n)}} \otimes \wedge^r W_{k(\n),S'}^* 
   \otimes \A_{H(\n)}^{\nu(\n)+1}, \\
\label{(7.4)}
\cR^\Artin_{k(\n)/k}(\epsilon_{k,S(\n),T,S'(\n)})
   &\in \lat{r}{}{U_{k(\n),S_0(\n)}} \otimes \wedge^r W_{k(\n),S'}^* 
   \otimes \A_{H(\n)}^{\nu(\n)+1}.
\end{align}

Suppose first that $|S-S'| \ge 2$.  Then $\epsilon_{k,S(\n),T,S'(\n)} = 0$ 
by Lemma \ref{xtra}, so \eqref{(7.4)} holds.
If $k$ has an archimedean place not in $S'$, then 
$\epsilon_{k(\n),S(\n),T,S'(\n)} = 0$ for $\n\ne 1$ by Lemma \ref{lem9.6}, 
so \eqref{(7.3)} holds.
If not, then $S$ contains two nonarchimedean primes; call one of them $v$ and 
let $S_0 := S - \{v\}$.  Since $v$ does not divide $\n$ and $S_0$ is still 
strictly larger than $S'$, all the 
hypotheses of Conjecture $\St(k(\n)/k,S_0(\n),T,S')$ are satisfied, so by 
Theorem \ref{ordvan} we have 
\begin{equation}
\label{(9.9)}
\Tw_{k(\n)/k}(\epsilon_{k(\n),S_0(\n),T,S'}) 
   \in \lat{r}{}{U_{k(\n),S_0(\n)}} \otimes \wedge^r W_{k(\n),S'}^* \otimes \A_{H(\n)}^{\nu(\n)}. 
\end{equation}
It follows directly from the defining properties (see for example
\cite[Proposition 3.6]{stark}) that 
$
\epsilon_{k(\n),S(\n),T,S'} = (1-\Frob_v^{-1})\epsilon_{k(\n),S_0(\n),T,S'}, 
$
so using \eqref{(9.9)}
\begin{multline*}
\Tw_{k(\n)/k}(\epsilon_{k(\n),S(\n),T,S'}) 
   = \Tw_{k(\n)/k}(\epsilon_{k(\n),S_0(\n),T,S'}) (1-\Frob_v^{-1}) \\
   \in \lat{r}{}{U_{k(\n),S_0(\n)}} \otimes \wedge^r W_{k(\n),S'}^* \otimes \A_{H(\n)}^{\nu(\n)+1}.
\end{multline*}
This is \eqref{(7.3)}.

Now suppose that $S'$ does not contain all archimedean places of $k$.
By Lemma \ref{lem9.6} we have $\epsilon_{k(\n),S(\n),T,S'} = 0$ for every $\n \ne 1$, 
so \eqref{(7.3)} holds.
If $S$ contains a nonarchimedean place then $|S-S'| \ge 2$, and we are in the 
case treated above.  So we may assume that $S$ is the set of all archimedean places.
Let $S' = \{v_1,\ldots,v_r\}$ and $\n = \prod_{i=1}^s \q_i$.  For $1 \le i \le s$ define
$\eta_i : U_{k,S(\n)} \to \A_{H(\n)}/\A_{H(\n)}^2$ to be the map given by the local Artin symbol
$$
\eta_i(u) := [u,k(\n)_{\q_i}/k_{\q_i}] - 1
$$
where $k(\n)_{\q_i}$ is the completion of $k(\n)$ at a prime above $\q_i$.  
Fix an expression 
$$
\epsilon_{k,S(\n),T,S'(\n)} = (u_1 \wedge \cdots \wedge u_{r+s}) 
   \otimes (v_1^* \wedge \cdots \wedge v_r^*  
   \wedge \q_1^* \wedge \cdots \wedge \q_s^*)
$$
with $u_i \in U_{k,S(\n)}$ (we have $\lat{r+s}{}{U_{k,S(\n)}} = \wedge^{r+s}U_{k,S(\n)}$ 
since $\Z[\Gamma] = \Z$).
Then concretely (ignoring the sign, which will not be important)
\begin{equation}
\label{(7.5)}
\cR^\Artin_{k(\n)/k}(\epsilon_{k,S(\n),T,S'(\n)}) 
   = \pm(\eta_1 \wedge \cdots \wedge \eta_s) (u_1 \wedge \cdots \wedge u_{r+s})
   \otimes (v_1^* \wedge \cdots \wedge v_r^*).
\end{equation}
In $\A_{H(\n)}/\A_{H(\n)}^2$, using the reciprocity law of global class field theory, we have 
for every $u \in U_{k,S(\n)}$
$$
\sum_{i=1}^s \eta_i(u) = \biggl(\prod_{\q\mid\n}[u,k(\n)_{\q}/k_{\q}]\biggr) - 1 
   = \prod_{w \nmid \n} [u,k(\n)_w/k_w]^{-1} - 1.
$$
If $w$ is nonarchimedean and does not divide $\n$, then $u$ is a unit at $w$ and 
$w$ is unramified in $k(\n)/k$, so $[u,k(\n)_w/k_w] = 1$.  If $w$ is archimedean, 
then $w$ splits completely in $k(\n)/k$, so again $[u,k(\n)_w/k_w] = 1$.  
Thus $\sum_{i=1}^s \eta_i : U_{k,S(\n)} \to \A_{H(\n)}/\A_{H(\n)}^2$ is the zero map, and we conclude 
using \eqref{(7.5)} that
\begin{multline*}
\cR^\Artin_{k(\n)/k}(\epsilon_{k,S(\n),T,S'(\n)}) 
   = \pm(\eta_1 \wedge \cdots \wedge \eta_s)(u_1 \wedge \cdots \wedge u_{r+s})
      \otimes (v_1^* \wedge \cdots \wedge v_r^*) \\
\textstyle
   = \pm(\eta_1 \wedge \cdots \wedge \eta_{s-1} \wedge (\sum_i\eta_i))(u_1 \wedge \cdots \wedge u_{r+s})
      \otimes (v_1^* \wedge \cdots \wedge v_r^*)
   = 0.
\end{multline*}
Thus \eqref{(7.4)} holds in this case as well, and the theorem follows.
\end{proof}

\section{Connection with Euler systems}
\label{essect}

Let $K/k$, $S'$, $S$, $T$, $\P$, $\N$, $K(\q)$, $K(\n)$, $S(\n)$, $S'(\n)$ 
be as in \S\ref{ov}, and let $\Gamma = \Gal(K/k)$.  Recall that $r := |S'|$.  

We assume further (by shrinking $K(\q)$ if necessary) 
that $[K(\q):K]$ is prime to $[K:k]$ for every $\q \in \P$.  It follows that for every $\q$ 
there is a unique extension $k(\q)/k$, totally ramified at $\q$ and unramified elsewhere, 
such that $K(\q) = Kk(\q)$.  Then if $k(\n)$ denotes the compositum 
of the $k(\q)$ for $\q$ dividing $\n$, we have $K(\n) = Kk(\n)$ for every $\n\in\N$, 
and 
\begin{equation}
\label{six.three}
\Gal(K(\n)/k) \cong \Gamma \times H(\n).
\end{equation}

Since all archimedean places split completely in $k(\q)/k$ for every $\q$, 
every $v \in S'$ splits completely in $K(\n)/k$ for every $\n$.  
Hence all hypotheses of Conjecture $\St(K(\n)/k,S(\n),T,S')$ are satisfied.

Fix an ordering $v_1,\ldots, v_r$ of the places in $S'$, 
and for each $i$ choose a place $w_i$ of 
the algebraic closure $\bar{k}$ above $v_i$.  Then for every $\n$, the element 
$$
\bw^*_\n := (w_1|_{K(\n)})^* \wedge \cdots \wedge (w_r|_{K(\n)})^*
$$
is a generator of the free, rank-one $\Z[\Gal(K(\n)/k)]$-module $\wedge^r W_{K(\n),S'}^*$.  
When $\n = 1$ we will write $\bw^*_K$ instead of $\bw^*_1$.

\begin{defn}
\label{xindef}
As in \S\ref{ov}, for every $\n\in\N$ we define
$$
\epsilon_{\n} := \epsilon_{K(\n),S(\n),T,S'} \in (\lat{r}{}{U_{K(\n),S(\n)}}) \otimes \wedge^r W_{K(\n),S'}^*
$$
to be the element predicted by Conjecture $\St(K(\n)/k,S(\n),T,S')$,
and we define 
$$
\xi_{\n} \in \lat{r}{}{U_{K(\n),S(\n)}} \subset \wedge^r U_{K(\n),S(\n)} \otimes \Q
$$
to be the unique element satisfying
$$
\xi_{\n} \otimes \bw^*_\n = \epsilon_{\n}.
$$
\end{defn}

\begin{prop}
\label{6.2}
If $\m, \n \in \N$, and $\m\mid\n$, then
$$
\bN_{K(\n)/K(\m)}^r \xi_{\n} = 
\prod_{\q\mid (\n/\m)} (1 - \Frob_\q^{-1}) \xi_{\m}.
$$
\end{prop}

\begin{proof}
This is \cite[Proposition 6.1]{stark}.
\end{proof}

By \eqref{six.three}, for every $\n\in\N$ we can view 
any $\Gal(K(\n)/k)$-module as a $\Gamma$-module.

Fix a rational prime $p$, not lying below any prime in $T$, and not 
dividing $[K:k]$.  
Fix also a character $\chi : \Gamma \to \bar{\Q}_p^\times$.
Let $\D := \Zp[\chi]$, the extension of $\Zp$ generated by the values of $\chi$.  
Since $p \nmid [K:k]$, the order of $\chi$ is prime to $p$ so $\D$ is unramified over $\Zp$. 
If $M$ is a $\Z[\Gamma]$-module, we let $M^\chi$ be the submodule of $M \otimes_{\Z} \D$ 
on which $\Gamma$ acts via $\chi$.  If $m \in M$, then 
\begin{equation}
\label{(6.2)}
m^\chi := \frac{1}{[K:k]}\sum_{\gamma\in\Gamma} m^\gamma \otimes \chi^{-1}(\gamma)
   \in M^\chi
\end{equation}
is the projection of $m$ into $M^\chi$.

Let $\Mchi := \Zp(1) \otimes \chi^{-1}$ denote a free $\D$-module of rank one on 
which $G_k$ acts via $\chi^{-1}$ times the cyclotomic character.

\begin{prop}
\label{h1u}
For every $\n \in \N$, Kummer theory gives Galois-equivariant isomorphisms 
$$
(K(\n)^\times)^{\chi} \cong H^1(k(\n),\Mchi),
$$
and if $\q$ is a prime of $k$
$$
((K \otimes_{k} k_\q)^\times)^\chi \cong H^1(k_\q,\Mchi).
$$
\end{prop}

\begin{proof}
This is a standard calculation; 
see for example \cite[\S6.1]{kolysys} or \cite[\S1.6.C]{eulersys}.
\end{proof}

\begin{thm}
\label{6.5}
Suppose that $r = 1$, and 
Conjecture $\St(K(\n)/k,S(\n),T,S')$ holds for every $\n \in \N$. 
Let $c_{\n} \in H^1(k(\n),\Mchi)$ 
denote the image of $\xi_{\n}^\chi$ under the Kummer map 
of Proposition \ref{h1u}. 
Then the collection 
$$
\{c_{\n} : \n\in\N\}
$$
is an Euler system for the $G_k$-representation $\Mchi$ 
in the sense of \cite[Definition 3.2.2]{kolysys} or \cite[\S9.1]{eulersys}.
\end{thm}

\begin{proof}
It follows from Proposition \ref{6.2} and \eqref{(6.2)} that 
if $\m, \n \in \N$ and $\m \mid \n$, then
$$
\bN_{K(\n)/K(\m)} \xi_{\n}^\chi = \prod_{\q\mid(\n/\m)} (1-\Frob_\q^{-1})\xi_{\m}^\chi.
$$
Translated to the elements $c_{\n}$ and $c_{\m}$, this is 
the defining property of an Euler system for $\Mchi$.
(Note that by the definition of $\N$ in \S\ref{ov}, we have $\chi(\q) = 1$ if $\q\mid\n$.)
\end{proof}

\begin{rem}
For general $r \ge 1$, the collection $\{c_{\n} : \n\in\N\}$ is not necessarily 
an Euler system in the sense of \cite[Definition 1.2.2]{pr}, because the elements 
$c_\n$ lie in $\lat{r}{}{H^1(k(\n),\Mchi)}$ rather than 
$\wedge^r H^1(k(\n),\Mchi)$.  This suggests that one might want to 
relax the definition of Euler system to allow elements to lie in the larger lattice.
\end{rem}

\section{Connection with Stark systems}
\label{kssect}

Let $K(\n)/K/k$, $\Gamma$, $S'$, $r$, $S$, $T$, $\P$, $\N$, $S(\n)$, $S'(\n)$, $\chi$ 
and $\Mchi$ be as in \S\ref{ov} and \S\ref{essect}. 
For $\n\in\N$ let $\nu(\n)$ denote the number of primes dividing $\n$.
We continue to suppose that $[K(\q):K]$ is prime to $[K:k]$ for every $\q\in\P$, 
and we now suppose in addition that 
\begin{equation}
\label{phyp}
p \nmid [K:k]\,\prod_{\lambda \in T_K}(\bN \lambda-1)
\end{equation}

Let $A$ denote the ring of integers of $K$, and for every $\n\in\N$ let $A_{S(\n)}$ 
denote the $S(\n)$-integers of $K$
$$
A_{S(\n)} := \{x \in K : \text{$\ord_\lambda(x) \ge 0$ for every $\lambda \notin S(\n)_K$}\}.
$$ 
Then $U_{K,S(\n)} = \{u \in A_{S(\n)}^\times : \text{$u \equiv 1 \pmod{\lambda}$ for every $\lambda \in T_K$}\}$.

\begin{lem}
\label{nope}
For every $\n\in\N$ we have $p \nmid [A_{S(\n)}^\times:U_{K,S(\n)}]$
\end{lem}

\begin{proof}
Reduction gives an injection 
$A_{S(\n)}^\times/U_{K,S(\n)} \hookto \oplus_{\lambda\in T_K} (A/\lambda)^\times$, 
so the lemma follows from our assumption \eqref{phyp}.
\end{proof}

\begin{lem}
\label{7.2}
For every $\n\in\N$ we have 
$(\lat{r+\nu(\n)}{}{U_{K,S(\n)}})^\chi = \wedge^{r+\nu(\n)}U_{K,S(\n)}^\chi$.
\end{lem}

\begin{proof}
By our choice of $T$, the group $U_{K,S(\n)}$ is torsion-free.  Since $p \nmid [K:k]$, 
we have $[K:k] \in \D^\times$, so $U_{K,S(\n)}\otimes \D$ is a projective 
$\D[\Gamma]$-module.  It now follows from Lemma \ref{2.5new} that
$$
\lat{r+\nu(\n)}{}{U_{K,S(\n)}} \otimes \D = \wedge^{r+\nu(\n)}U_{K,S(\n)}\otimes \D.
$$
Taking $\chi$-components proves the lemma.
\end{proof}

Define
$$
\Np = \{\n\in\N : \text{$\n$ is prime to $p$}\}.
$$
For $\n\in\Np$ recall that $H(\n) := \Gal(K(\n)/K)$, and
$\A_{H(\n)} \subset \D[H(\n)]$ is the augmentation ideal.
Define an ideal $I_\n \subset \D$ by
$$
I_\n := \sum_{\q\mid\n}([k(\q):k]\D)
$$ 
(with the convention $I_1 = 0$).
Let $W_{K,\n}$ denote the free abelian group on the set of primes of $K$ dividing $\n$, 
so $W_{K,S'(\n)} = W_{K,S'} \oplus W_{K,\n}$ and 
\begin{equation}
\label{(8.3)}
\wedge^{r+\nu(\n)} W^*_{K,S'(\n)} = \wedge^{\nu(\n)} W^*_{K,\n} \otimes \wedge^r W^*_{K,S'}.
\end{equation}

\begin{defn}
For every $\n\in\Np$, define
$$
Y_\n := \wedge^{r+\nu(\n)}U_{K,S(\n)}^\chi \otimes \wedge^{\nu(\n)} (W^*_{K,S'(\n)})^\chi \otimes (\D/I_\n).
$$
If $\m \mid \n$, we define a map 
$$
\Psi_{\n,\m} : Y_\n \too Y_\m \otimes (\D/I_\n)
$$
as follows.  
Fix a prime factorization $\n/\m = \q_1 \cdots \q_t$ and for each $i$ fix a prime 
$\fQ_i$ of $K$ above $\q_i$.  Define $\psi_i \in U_{K,S(\n)}^*$ by 
$\psi_i(u) = \sum_{\gamma\in\Gamma}\ord_{\fQ_i}(u^\gamma)\gamma^{-1}$.
By Definition \ref{wdg1} we get a map 
$$
\psi_1 \wedge \cdots \wedge \psi_t : \wedge^{r+\nu(\n)}U_{K,S(\n)}^\chi  \too 
   \wedge^{r+\nu(\m)}U_{K,S(\n)}^\chi 
$$
and by \cite[Lemma 5.1]{stark} or \cite[Proposition A.1]{gen.kolysys} 
the image of this map is contained in $\wedge^{r+\nu(\m)}U_{K,S(\m)}^\chi$.
Further, viewing $\fQ_1 \wedge \cdots \wedge \fQ_t$ as a generator of $\wedge^t W_{\n/\m}$
the map
\begin{multline}
\label{(9.5)}
(\psi_1 \wedge \cdots \wedge \psi_t) \otimes (\fQ_1 \wedge \cdots \wedge \fQ_t) \\
   : \wedge^{r+\nu(\n)}U_{K,S(\n)}^\chi \otimes \wedge^t (W_{\n/\m}^*)^\chi  \too 
   \wedge^{r+\nu(\m)}U_{K,S(\m)}^\chi 
\end{multline}
is independent of the choice of the $\fQ_i$ and the order of the $\q_i$.  
Now we define $\Psi_{\n,\m}$ to be the composition
\begin{align*}
Y_\n &= \wedge^{r+\nu(\n)}U_{K,S(\n)}^\chi 
   \otimes \wedge^{\nu(\n)} (W^*_{K,S'(\n)})^\chi \otimes (\D/I_\n) \\
   &\xrightarrow{\;\sim\;} \wedge^{r+\nu(\n)}U_{K,S(\n)}^\chi 
      \otimes \wedge^{\nu(\m)} (W^*_{K,S'(\m)})^\chi 
      \otimes \wedge^{\nu(\n/\m)} (W^*_{\n/\m})^\chi \otimes (\D/I_\n) \\
   &\too \wedge^{r+\nu(\m)}U_{K,S(\m)}^\chi 
   \otimes \wedge^{\nu(\m)} (W^*_{K,S'(\m)})^\chi \otimes (\D/I_\n) = Y_\m \otimes (\D/I_\n),
\end{align*}
where the last map is induced by \eqref{(9.5)}.  Note that $\Psi_{\n,\m}$ is the map $\Phi$ of 
\cite[\S5]{stark}.
\end{defn}

Using Lemma \ref{7.2} we can view $\epsilon_{\n}^\chi \in Y_\n$, where $\epsilon_{\n}$ is 
the element of Definition \ref{xindef} predicted by Conjecture $\St(K(\n)/k,S(\n),T,S')$.
The following lemma allows us to apply the results of \cite{gen.kolysys} 
to the family of $Y_\n$.

\begin{lem}
\label{8.1}
The modules $Y_\n$ and the maps $\Psi_{\n,\m}$ defined above are the same as the $Y_\n$ 
and $\Psi_{\n,\m}$ of \cite[Definition {\yndef2}]{gen.kolysys} for the 
Galois representation $\Mchi$. 
\end{lem}

\begin{proof}
The proof is an exercise, using the natural Kummer theory isomorphisms 
$(K^\times)^\chi \cong H^1(k,\Mchi)$ and 
$((K \otimes k_v)^\times)^\chi \cong H^1(k_v,\Mchi)$ 
for places $v$ of $k$ (Proposition \ref{h1u}), 
along with Lemma \ref{nope}.
\end{proof}

\begin{defn}
As in \cite[Definition {\yndef2}]{gen.kolysys} we say that a collection 
$$
\{\sigma_\n \in Y_\n : \n\in\Np\}
$$
is a {\em Stark system of rank $r$} if
$$
\Psi_{\n,\m}(\sigma_\n) = \sigma_\m \otimes 1 \in Y_\m \otimes (\D/I_\n)
   \quad\text{whenever $\m\mid\n \in \Np$.}
$$
Let $\SS_r(\Mchi)$ denote the $\D$-module of Stark systems of rank $r$.
\end{defn}

Suppose for the rest of this section that Conjecture $\St(K/k,S(\n),T,S'(\n))$ holds 
for every $\n\in\N$.  
Recall that $\bw_K^*$ is the generator of $\wedge^r W^*_{K,S'}$ fixed at the beginning 
of \S\ref{essect}, and $\one$ denotes the trivial character of $\Gamma$.

\begin{defn}
\label{sd}
For $\n\in\N$ let 
$\delta_{\n} \in (\wedge^{r+\nu(\n)}{U_{K,S(\n)}}) \otimes \wedge^{\nu(\n)} (W^*_{K,\n})$ 
be the unique element such that 
$$
\delta_{\n} \otimes \bw_K^* := \epsilon_{K,S(\n),S'(\n)} 
   \in (\wedge^{r+\nu(\n)}{}{U_{K,S(\n)}}) \otimes \wedge^{r+\nu(\n)} W^*_{K,S'(\n)}
$$
is the element predicted by Conjecture $\St(K/k,S(\n),T,S'(\n))$, using 
the identifications of Lemma \ref{7.2} and \eqref{(8.3)}.  
Then 
$$
\delta_\n^\chi \otimes 1 \in 
   \wedge^{r+\nu(\n)}U_{K,S(\n)}^\chi \otimes \wedge^{\nu(\n)} (W^*_{K,S'(\n)})^\chi \otimes (\D/I_\n) = Y_\n, 
$$
and we denote by $\stsys^\chi$ the collection $\{\delta_\n^\chi \otimes 1 \in Y_\n : \n\in\Np\}$.
\end{defn}

\begin{prop}
\label{7.4}
We have $\stsys^\chi \in \SS_r(\Mchi)$, i.e., $\stsys^\chi$ 
is a Stark system of rank $r$.
\end{prop}

\begin{proof}
If $\n\in\N$ and $\m \mid \n$, then $\Psi_{\n,\m}(\delta_{\n}^\chi\otimes 1) = \delta_{\m}^\chi \otimes 1$
by \cite[Proposition 5.2]{stark}.
\end{proof}

Let $r(\chi,S)$ be as in Definition \ref{4.5}.

\begin{lem} 
\label{8.7}
\begin{enumerate}
\item
If $r(\chi,S) > r$, then $\delta_\n^\chi = 0$ for every $\n\in\N$.
\item
If $r(\chi,S) = r$, then
$\delta_{\n}^\chi$ is a nonzero element of the free, rank-one $\D$-module 
$\wedge^{r+\nu(\n)} U_{K,S(\n)}^\chi \otimes \wedge^{\nu(\n)} (W^*_{K,\n})^\chi$.
\end{enumerate}
\end{lem}

\begin{proof}
The $\ZGamma$-module $W^*_{K,\n}$ is free of rank $\nu(\n)$.  
By the basic properties of Conjecture $\St(K/k,S(\n),T,S'(\n))$ we have 
$$
\delta_\n^\chi \ne 0 \iff r(\chi,S(\n)) = r + \nu(\n) \iff r(\chi,S) = r,
$$
and if these equivalent conditions hold then $U_{K,S(\n)}^\chi$ is free of rank 
$r+\nu(\n)$ over $\D$.  
The lemma follows.
\end{proof}

\section{The case $r=1$}
\label{r=1}

Keep the setting and notation of the previous two sections.  In this section 
we will prove (Theorem \ref{9.7}) a part of Conjecture \ref{conj}(ii) when $r=1$.
The idea of the proof is as follows.  

The Stark system $\stsys^\chi$ of \S\ref{kssect} gives rise (via 
an explicit construction) to a Kolyvagin system for $\Mchi$.  
When $r=1$, the Euler system of Stark elements of Theorem \ref{6.5} also gives rise
(via an explicit construction) to a Kolyvagin system for $\Mchi$.  The $\D$-module of Kolyvagin 
systems for $\Mchi$ is free of rank one, and the two Kolyvagin systems agree when $\n = 1$ 
by construction.  Hence the two Kolyvagin systems agree for every $\n$, and unwinding the two 
explicit constructions shows that the agreement for $\n$ is equivalent to the ``$(p,\chi)$-part'' 
of Conjecture \ref{conj}(ii) for $(K(\n)/K/k,S(\n),T,S')$.

As in \S\ref{ov}, if $\m\mid\n$ we can view $H(\m)$ as both a subgroup and a quotient 
of $H(\n)$, and $\pi_\m : H(\n) \onto H(\m) \hookto H(\n)$ is the projection map.

\begin{defn}
\label{Mdef}
If $\n \in \N$ and $\d = \prod_{i=1}^t \q_i$ divides $\n$, let $M_{\n,\d} = (m_{ij})$ 
be the $t \times t$ matrix with entries in $\A_{H(\n)}/\A_{H(\n)}^2$
$$
m_{ij} = 
\begin{cases}
\pi_{\n/\d}(\Frob_{\q_i}-1) & \text{if $i = j$},\\
\pi_{\q_j}(\Frob_{\q_i}-1) & \text{if $i \ne j$},
\end{cases}
$$
and define
$$
\DD_{\n,\d} := \det(M_{\n,\d}) \in \A_{H(\n)}^t/\A_{H(\n)}^{t+1}
$$
(this is independent of the ordering of the prime factors of $\d$).
By convention we let $\DD_{\n,1} = 1$.
For $\n \in\N$, let $\BB_\n$ denote the cyclic group
$$
\BB_\n := \{\textstyle\prod_{\q\mid\n}(\gamma_\q - 1) : 
   \gamma_\q \in H(\q)\} \subset \A_{H(\n)}^{\nu(\n)}/\A_{H(\n)}^{\nu(\n)+1}.
$$
\end{defn}
By \cite[Proposition 4.2]{darmon.conj}, $\BB_\n$ is a direct summand of 
$\A_{H(\n)}^{\nu(\n)}/\A_{H(\n)}^{\nu(\n)+1}$.

Let $\KS_r(\Mchi)$ denote the $\D$-module of Kolyvagin systems of rank $r$ for $\Mchi$ 
(with the natural Selmer structure of \cite[\S\exgm]{gen.kolysys}) as 
defined in \cite[\S\ksp]{gen.kolysys} (see also \cite[\S\exgm]{gen.kolysys} and 
\cite[\S3.1 and \S6.1]{kolysys}).  A Kolyvagin system of rank $r$ 
for $\Mchi$ is a collection 
$$
\{\kappa_\n \in \wedge^r U_{K,S(\n)}^\chi \otimes \BB_\n : \n \in \Np\}
$$
satisfying properties that we do not need to review here.  We are identifying 
$\otimes_{\q\mid\n}H(\q)$ with $\BB_\n$ via 
$\otimes_\q \gamma_\q \mapsto \prod_\q(\gamma_\q-1)$.

\begin{defn}
\label{def9.3}
For $\n\in\N$ let 
$\delta_\n^\chi \in \wedge^{r+\nu(\n)}U_{K,S(\n)}^\chi 
   \otimes \wedge^{\nu(\n)}(W_{K,\n}^*)^\chi$ 
be as in Definition \ref{sd}, and define
$$
\beta^\St_\n := \sum_{\d\mid\n} \cR_{K(\d)/K}^\Artin(\delta_\d^\chi)\cdot\DD_{\n,\n/\d}
   \in \wedge^r U_{K,S(\n)}^\chi \otimes \A_{H(\n)}^{\nu(\n)}/\A_{H(\n)}^{\nu(\n)+1}.
$$
\end{defn}

\begin{prop}
\label{9.3}
For $\n\in\N$ we have $\beta^\St_\n \in \wedge^r U_{K(\n),S(\n)}^\chi \otimes \BB_\n$, and 
the collection
$$
\boldsymbol{\beta}^\St := \{\beta^\St_\n : \n\in\Np\}
$$
is a Kolyvagin system of rank $r$ for $\Mchi$.
\end{prop}

\begin{proof}
In the special case where $k = \Q$, $S' = \{\infty\}$, 
and $\chi$ is an even quadratic character, this is 
\cite[Theorem 8.7 and Proposition 6.5]{darmon.conj}.  The proof in general is similar.
The general case is also proved by Sano in \cite[\S4]{sano2} (what we call a Stark system 
is called a unit system in \cite{sano2}).
\end{proof}

For the rest of this section we assume that $r = 1$, i.e., $S'$ consists of a 
single archimedean place.  Since $r = 1$, the Stark unit Euler system of 
Theorem \ref{6.5} gives rise, via the map of \cite[Theorem 3.2.4]{kolysys}, to a 
Kolyvagin system of rank one
$$
\stesys = \{\stek_\n : \n\in\Np\} \in \KS_1(\Mchi).
$$ 
(The results of \cite{kolysys} are stated only 
for $k = \Q$, but the proofs in the general case are the same; see \cite{gen.kolysys}.)

\begin{prop}
\label{9.2}
Suppose $\n\in\Np$.   
Under the restriction map $K^\times \to K(\n)^\times$ and the inclusion 
$\BB_\n \subset \A_{H(\n)}^{\nu(\n)}/\A_{H(\n)}^{\nu(\n)+1}$, 
with $\xi_\n$ as in Definition \ref{xindef} we have
$$
\stek_\n \mapsto \sum_{\d\mid\n} \Tw_{K(\d)/K}(\xi_\d^\chi)\cdot\DD_{\n,\n/\d}
   \in U_{K(\n),S(\n)}^\chi \otimes \A_{H(\n)}^{\nu(\n)}/\A_{H(\n)}^{\nu(\n)+1}.
$$
\end{prop}

\begin{proof}
Note that $\Tw_{K(\d)/K}(\xi_\d^\chi)$ lies in 
$U_{K(\n),S(\n)}^\chi \otimes \A_{H(\n)}^{\nu(\d)}/\A_{H(\n)}^{\nu(\d)+1}$ by 
Theorem \ref{ordvan} and Lemma \ref{7.2}, and $\DD_{\n,\n/\d}$ lies in 
$\A_{H(\n)}^{\nu(\n/\d)}/\A_{H(\n)}^{\nu(\n/\d)+1}$ by definition.

In the special case where $k = \Q$ and $\chi$ is a real quadratic character, this is 
\cite[Theorem 7.2 and Proposition 6.5]{darmon.conj}.  The proof in general is the same.
The general case also follows from calculations of Sano \cite[\S3]{sano2}.
\end{proof}

\begin{thm}
\label{9.5}
If $\chi \ne \one$ then for every $\n\in\N$ we have $\stek_\n = \beta^\St_\n$.
\end{thm}

\begin{proof}
Let $r(\chi,S)$ be as in Definition \ref{4.5}, and suppose first that $r(\chi,S) = 1$.  
We have $\stesys, \boldsymbol{\beta}^\St \in \KS_1(\Mchi)$.  Since $\chi \ne \one$ 
and $K$ contains no nontrivial $p$-th roots of unity by Lemma \ref{nope}, all the hypotheses of 
\cite[\S3.5]{kolysys} hold, so $\KS_1(\Mchi)$ is a free $\D$-module 
of rank one by \cite[Theorem 5.2.10]{kolysys}.
We have $\beta^\St_1 = \delta_1^\chi = \xi_1^\chi = \stek_1$ by definition, 
and by Lemma \ref{8.7}(ii) this is a nonzero element of the free, rank-one $\D$-module 
$U_{K,S}^\chi$.  Hence $\boldsymbol{\beta}^\St = \stesys$, i.e., $\stek_\n = \beta^\St_\n$ 
for every $\n\in\Np$.

Now suppose $r(\chi,S) > 1$.  By Lemma \ref{8.7}(i), we have $\delta_\n^\chi = 0$ 
for every $\n$, so $\beta^\St_\n = 0$ for every $\n$.
Since $\stek_1 = 0$, the finiteness of the ideal class group together with 
\cite[Theorem \korankk(iv) and Proposition \thirteenone]{gen.kolysys}
(see also \cite[Theorem 5.2.12]{kolysys}) shows that $\stesys = 0$, 
i.e., $\stek_\n = 0$ for every $\n\in\Np$.

It remains to show that $\stek_\n = \beta^\St_\n \in U_{K,S(\n)}^\chi \otimes \BB_\n$ 
when $\n\in\N-\Np$.  But the exponent of the cyclic group $\BB_\n$ 
is the greatest common divisor of the $|H(\q)|$ for $\q$ dividing $\n$.
If $\q\mid p$ then (since $K(\q)$ is tamely ramified by definition) $H(\q)$ 
has order prime to $p$.  Hence $\BB_\n$ has order prime to $p$ if $\n \in \N-\Np$, 
so $\BB_\n \otimes \D = 0$ and $\stek_\n = \beta^\St_\n = 0$.
This completes the proof.
\end{proof}

\begin{thm}
\label{9.6}
Suppose that $|S'| = 1$, that Conjectures $\St(K/k)$ and $\St(K(\n)/k)$ hold for every $\n$, and 
that at least one of the following holds:
\begin{enumerate}
\renewcommand{\theenumi}{(\alph{enumi})}
\item
$\chi \ne \one$, 
\item
$\chi = \one$ and $|S-S'| \ge 2$,
\item
$\chi = \one$ and $k$ has more than one archimedean place,
\end{enumerate}
Then for every $\n\in\N$,
$$
\Tw_{K(\n)/K}(\epsilon_{K(\n),S(\n),T,S'}^\chi) = 
   \cR_{K(\n)/K}^\Artin(\epsilon_{K,S(\n),T,S'(\n)}^\chi)
$$
in $U_{K(\n),S(\n)} \otimes_{\Gal(K(\n)/k)} W_{K(\n),S'}^* \otimes \A_{H(\n)}^{\nu(\n)}/\A_{H(\n)}^{\nu(\n)+1}$.
In other words, the $(p,\chi)$ part of Conjecture \ref{conj}(ii) holds for $(K(\n)/K/k,S(\n),T,S')$.
\end{thm}

\begin{proof}
If $\chi \ne \one$, then this follows directly from Theorem \ref{9.5} by induction on $\n$, 
using Proposition \ref{9.2} and Definition \ref{def9.3} for the induction.
If $\chi = \one$, then this is Theorem \ref{K=kthm}.
\end{proof}

Let $\Sigma = \Sigma(K/k,T)$ be the set of  primes dividing 
$[K:k]\prod_{\lambda\in T_K}(\bN\lambda - 1)$.

\begin{thm}
\label{9.7}
Suppose that $|S'| = 1$, that Conjectures $\St(K/k)$ and $\St(K(\n)/k)$ hold for every $\n$, 
and that either $k$ has more than one archimedean place or $|S| \ge 3$.  
Then Conjecture \ref{conj}(ii) holds for $(K(\n)/K/k,S(\n),T,S')$ away from $\Sigma$, i.e., 
for every $p \notin \Sigma$ the leading term formula holds if we tensor with $\Zp$.
\end{thm}

\begin{proof}
We can apply 
Theorem \ref{9.6} for every prime $p \notin \Sigma$, and every character $\chi$ of 
$\Gamma$.  Summing the conclusion of Theorem \ref{9.6} over all $\chi$ gives 
the equality of Conjecture \ref{conj}(ii) tensored with $\D$.
\end{proof}

\section{Evidence in the case of general $r$}
Keep the notation of the previous sections.
When $r > 1$, the proof of \S\ref{r=1} breaks down.  Namely, 
the elements $\xi_\n$ of Definition \ref{xindef} naturally form an Euler 
system of rank $r$, but when $r > 1$ we do not know how to use this Euler system to 
produce a Kolyvagin system of rank $r$.  
However, using ideas of \cite[\S6]{stark} and \cite{kazim} we define a family 
of ``projectors'' $\Phi$, each of which maps the collection $\{\xi_\n^\chi\}$ to an 
Euler system $\boldsymbol{\xi}^\St_\Phi$ of rank one, and maps the rank-$r$ Kolyvagin system 
$\boldsymbol{\beta}^\St$ to a rank-one Kolyvagin system $\boldsymbol{\beta}^\St_\Phi$.
We can associate to $\boldsymbol{\xi}^\St_\Phi$ a 
Kolyvagin system $\stesys_\Phi$ of rank one, and the arguments of \S\ref{r=1} 
will show that $\boldsymbol{\beta}^\St_\Phi = \stesys_\Phi$.  
Unwinding the definitions, this shows that the $\Phi$-projection of the 
leading term formula of Conjecture \ref{conj} holds.

For this section we make the extra assumptions that 
\begin{itemize}
\item
$S$ contains no primes above $p$,
\item
$k$ is totally real of degree $r$ and $S'$ is the set of its archimedean places,
\item
Leopoldt's conjecture holds for $K$.
\end{itemize}
In particular $K$ is totally real and $K/k$ is unramified above $p$.

\begin{defn}
\label{11.1}
For every $\n \in \Np$ let $V_{K(\n)}$ denote the $p$-adic completion of the local units 
of $K(\n) \otimes \Qp$, and $V_{K(\n)}^* := \Hom_{\Gal(K(\n)/k)}(V_{K(\n)},\Zp[\Gal(K(\n)/k)])$.  
If $\phi \in V_{K(\n)}^*$, then $\tilde\phi$ will denote the composition
$$
\tilde\phi : U_{K(\n),S(\n)} \too V_{K(\n)} \too \Zp[\Gal(K(\n)/k)].
$$

Define
$
V_\infty^* := \varprojlim V_{K(\n)}^*,
$
where the inverse limit is taken with respect to the maps 
$V_{K(\n\q)}^* \to V_{K(\n)}^*$ induced by 
$$
V_{K(\n)} \subset V_{K(\n\q)}, \quad
   \Zp[\Gal(K(\n\q)/k)]^{\Gal(K(\n\q)/K(\n))} = \Zp[\Gal(K(\n)/k)].
$$
If $\Phi := \phi_1\wedge\ldots\wedge\phi_{r-1} \in \wedge^{r-1}V_\infty^*$, 
with $\phi_i \in V_\infty^*$, let 
$$
\phi_{i,K(\n)} : V_{K(\n)} \too \Zp[\Gal(K(\n)/k)]
$$ 
denote the projection of $\phi_i$ to $V_{K(\n)}^*$, let
$$
\tilde\Phi_{K(\n)} := \tilde\phi_{1,K(\n)}\wedge\cdots\wedge\tilde\phi_{r-1,K(\n)} 
   : \lat{r}{}{U_{K(\n),S(\n)}} \too U_{K(\n),S(\n)}
$$
be the map of Definition \ref{wdg1} (combined with Lemmas \ref{2.5new} and \ref{rs}), and let 
$$
\cL_\Phi := \cap_i\ker(\phi_{i,K}) \subset V_{K}.
$$
Using the identification $V_K^\chi \subset \oplus_{\p\mid p}H^1(k_\p,\Mchi)$ of 
Proposition \ref{h1u}, we define a Selmer structure 
(see \cite[Definition \selmerdef]{gen.kolysys} or \cite[Definition 2.1.1]{kolysys}) 
$\cF_\Phi$ on $\Mchi$ by modifying the natural 
Selmer structure $\cF_{\unr}$ of \cite[\S\exgm]{gen.kolysys} at primes 
above $p$, namely we set 
$$
\oplus_{\p\mid p}H^1_{\cF_\Phi}(k_\p,\Mchi) 
   := \cL_\Phi^\chi \subset V_K^\chi \subset \oplus_{\p\mid p}H^1(k_\p,\Mchi).
$$
\end{defn}

Let $\xi_\n\in\lat{r}{}{U_{K(\n),S(\n)}}$ be as in Definition \ref{xindef}, 
and recall the Kolyvagin system 
$\boldsymbol{\beta}^\St = \{\beta^\St_\n : \n\in\Np\} \in \KS_r(\Mchi)$ of 
Definition \ref{def9.3} and Proposition \ref{9.3}.

\begin{prop}
\label{11.2}
Suppose $\Phi := \phi_1\wedge\ldots\wedge\phi_{r-1} \in \wedge^{r-1}V_\infty^*$.
\begin{enumerate}
\item
The collection 
$\{\tilde\Phi_{K(\n)}(\xi_\n^\chi) \in U_{K(\n),S(\n)}^\chi : \n\in\Np\}$
is an Euler system of rank one for the representation $\Mchi$.
\item
Let $\stesys_\Phi = \{\stek_{\Phi,\n} : \n\in\Np\} \in \KS_1(\Mchi)$ 
be the Kolyvagin system of rank one attached to the Euler system of (i) by 
\cite[Theorem 3.2.4]{kolysys}.  Then $\stesys_\Phi \in \KS_1(\Mchi,\cF_\Phi)$, 
where $\cF_\Phi$ is the Selmer structure of Definition \ref{11.1}.
\item
The collection
$\boldsymbol{\beta}^\St_\Phi := \{\tilde\Phi_K(\beta^\St_\n) : \n\in\Np\}$
is a Kolyvagin system of rank one for $(\Mchi,\cF_\Phi)$.
\end{enumerate}
\end{prop}

\begin{proof}
The first assertion is proved in \cite[Proposition 6.6]{stark}.
Both (i) and (ii) are proved in \cite[Proposition 2.2 and Theorem 2.19]{kazim}.
Assertion (iii) follows from Proposition \ref{9.3} by direct calculation.
\end{proof}

\begin{prop}
\label{11.3}
Suppose $\Phi := \phi_1\wedge\ldots\wedge\phi_{r-1} \in \wedge^{r-1}V_\infty^*$ and $\n\in\Np$.   
Under the restriction map $K^\times \to K(\n)^\times$ and the inclusion 
$\BB_\n \subset \A_{H(\n)}^{\nu(\n)}/\A_{H(\n)}^{\nu(\n)+1}$, 
we have
$$
\stek_{\Phi,\n} \mapsto \sum_{\d\mid\n} \Tw_{K(\d)/K}(\tilde\Phi_{K(\d)}(\xi_\d^\chi))\cdot\DD_{\n,\n/\d}
   \in U_{K(\n),S(\n)}^\chi \otimes \A_{H(\n)}^{\nu(\n)}/\A_{H(\n)}^{\nu(\n)+1}.
$$
\end{prop}

\begin{proof}
The proof is similar to Proposition \ref{9.2}, or see \cite[\S3]{sano2}.
\end{proof}

\begin{thm}
\label{11.4}
Suppose $\Phi := \phi_1\wedge\ldots\wedge\phi_{r-1} \in \wedge^{r-1}V_\infty^*$.
If $\chi \ne \one$ then for every $\n\in\N$ we have 
$\stek_{\Phi,\n} = \tilde\Phi_K(\beta^\St_\n)$.
\end{thm}

\begin{proof}
The proof is similar to that of Theorem \ref{9.5}.

Let $r(\chi,S)$ be as in Definition \ref{4.5}, and suppose first that $r(\chi,S) = r$ 
and $\phi_{1,K}, \dots, \phi_{r-1,K}$ are $\Zp$-linearly independent.
We have $\stesys_\Phi, \boldsymbol{\beta}^\St_\Phi \in \KS_1(\Mchi,\cF_\Phi)$ 
by Proposition \ref{11.2}.  
The core rank of $(\Mchi,\cF_\Phi)$ is one by \cite[Proposition 1.8]{kazim}.
All the hypotheses of 
\cite[\S3.5]{kolysys} hold, so $\KS_1(\Mchi,\cF_\Phi)$ is a free $\D$-module 
of rank one by \cite[Theorem 5.2.10]{kolysys}.
We have 
$\tilde\Phi_K(\beta^\St_1) = \tilde\Phi_K(\delta_1^\chi) 
   = \tilde\Phi_K(\xi_1^\chi) = \stek_{\Phi,1}$ 
by definition, and it follows from 
Lemma \ref{8.7}(ii), our assumption on the independence of the $\phi_{i,K}$, 
and Leopoldt's conjecture that this has infinite order in $\cL_\Phi^\chi$.  
Hence $\boldsymbol{\beta}^\St_\Phi = \stesys_\Phi$, i.e., 
$\stek_{\Phi,\n} = \tilde\Phi_K(\beta^\St_\n)$ 
for every $\n\in\Np$.

Now suppose that either $r(\chi,S) > r$ or the $\phi_{i,K}$ are linearly 
dependent.  In the former case Lemma \ref{8.7}(i) shows that $\delta_\n^\chi = 0$ 
for every $\n$, and in the latter case $\tilde\Phi_K = 0$, so in either case 
$\tilde\Phi_K(\beta^\St_\n) = 0$ for every $\n$.
Since $\stek_{\Phi,1} = 0$, the finiteness of the ideal class group together with 
\cite[Theorem \korankk(iv) and Proposition \thirteenone]{gen.kolysys}
(see also \cite[Theorem 5.2.12]{kolysys}) and Leopoldt's conjecture 
(see \cite[Remark 1.7]{kazim}) shows that $\stesys_\Phi = 0$, 
i.e., $\stek_{\Phi,\n} = 0$ for every $\n\in\Np$.

It remains to show that $\stek_\n = \beta^\St_\n \in U_{K,S(\n)}^\chi \otimes \BB_\n$ 
when $\n\in\N-\Np$.  This follows exactly as in the proof of Theorem \ref{9.5}, 
since $\BB_\n$ has order prime to $p$ if $\n \in \N-\Np$.
This completes the proof.
\end{proof}

\begin{thm}
\label{11.5}
Suppose that Conjectures $\St(K/k)$ and $\St(K(\n)/k)$ hold for every $\n$, and 
either $\chi \ne \one$ or $|S-S'| \ge 2$.  Then for every 
$\Phi := \phi_1\wedge\ldots\wedge\phi_{r-1} \in \wedge^{r-1}V_\infty^*$ and every 
$\n\in\N$,
$$
\tilde\Phi_{K(\n)}(\Tw_{K(\n)/K}(\epsilon_{K(\n),S(\n),T,S'}^\chi)) = 
   \tilde\Phi_{K(\n)}(\cR_{K(\n)/K}^\Artin(\epsilon_{K,S(\n),T,S'(\n)}^\chi))
$$
in $U_{K(\n),S(\n)} \otimes_{\Gal(K(\n)/k)} W_{K(\n),S'}^* \otimes \A_{H(\n)}^{\nu(\n)}/\A_{H(\n)}^{\nu(\n)+1}$.
In other words, the $\tilde\Phi_{K(\n)}$-pro\-jec\-tion of the 
$\chi$ part of Conjecture \ref{conj}(ii) holds for $(K(\n)/K/k,S(\n),T,S')$.
\end{thm}

\begin{proof}
If $\chi \ne \one$, then this follows directly from Theorem \ref{11.4} by induction on $\n$, 
using Proposition \ref{11.3} and Definition \ref{def9.3} for the induction.
If $\chi = \one$, then this is Theorem \ref{K=kthm}.
\end{proof}

Let $\Sigma = \Sigma(K/k,S,T)$ be the set of rational primes dividing
$$
[K:k]\prod_{\lambda\in S - S'}\bN\lambda\prod_{\lambda\in T_K}(\bN\lambda - 1).
$$

\begin{thm}
\label{11.6}
Suppose that $k$ is totally real, $S'$ is the set of all archimedean places of $k$, 
Conjectures $\St(K/k)$ and $\St(K(\n)/k)$ hold for every $\n$, Leopoldt's conjecture 
holds for $K$, and $|S-S'| \ge 2$.  
Then for every $p \notin \Sigma$, every $\Phi \in \wedge^{r-1}V_\infty^*$, and every 
$\n\in\N$, we have
$$
\tilde\Phi_{K(\n)}(\Tw_{K(\n)/K}(\epsilon_{K(\n),S(\n),T,S'})) = 
   \tilde\Phi_{K(\n)}(\cR_{K(\n)/K}^\Artin(\epsilon_{K,S(\n),T,S'(\n)})).
$$
In other words, for every $\Phi \in \wedge^{r-1}V_\infty^*$ the leading term formula of
Conjecture \ref{conj}(ii) holds for $(K(\n)/K/k,S(\n),T,S')$ after applying $\tilde\Phi_{K(\n)}$.
\end{thm}

\begin{proof}
We can apply 
Theorem \ref{11.5} for every prime $p \notin \Sigma$, and every character $\chi$ of 
$\Gamma$.  Summing the conclusion of Theorem \ref{11.5} over all $\chi$ gives 
the $\tilde\Phi_{K(\n)}$-projection of Conjecture \ref{conj}(ii) tensored with $\D$.
\end{proof}

\appendix
\section{Exterior algebras and determinants}
\label{wedges}

Let 
$\D$ be an integral domain with field of fractions $F$, and let $R = \D[\Gamma]$
with a finite abelian group $\Gamma$.  

If $M$ is an $R$-module, we let $M^* := \Hom_R(M,R)$, and $\mt{M}$ 
will denote the image of $M$ in $M \otimes_\D F$.
If $\rho \in R$, then $M[\rho]$ denotes the kernel of multiplication by 
$\rho$ in $M$.

Fix for this appendix an $R$-module $M$ of finite type.

\begin{defn}
\label{wdg1}
If $r \ge 0$, then $\wedge^r M$ (or $\wedge_R^r M$, if we need to emphasize the ring $R$) 
will denote the $r$-th exterior power of $M$ in the category of $R$-modules,
with the convention that $\wedge^0M = R$.  
If $\psi \in M^*$ and $r \ge 1$, we view $\psi\in\Hom(\wedge^rM,\wedge^{r-1}M)$ by
$$
\psi(m_1 \wedge \cdots \wedge m_r) := \sum_{i=1}^r (-1)^{i+1}\psi(m_i)
   (m_1 \wedge \cdots \wedge m_{i-1} \wedge m_{i+1} \wedge \cdots \wedge m_r).
$$
If $\bpsi\in\wedge^sM^*$ with $s \le r$, we view $\psi\in\Hom(\wedge^rM,\wedge^{r-s}M)$ by
$$
(\psi_1 \wedge \cdots \wedge \psi_s)(\bm) :=  \psi_s \circ \psi_{s-1}\circ \cdots \circ \psi_1(\bm).
$$
\end{defn}

In particular 
\begin{equation}
\label{detform}
(\phi_1 \wedge \cdots \wedge \phi_r)  (m_1 \wedge \cdots \wedge m_r)
   = \det(\phi_i(m_j)).
\end{equation}

\begin{defn}
\label{mylat}
For every $r \ge 0$, define 
$$
\lat{r}{R}{M} := \{\bm \in \mt{\wedge^r M} : 
   \text{$\bpsi(\bm) \in R$ for every $\bpsi \in \wedge^{r}M^*$}\}.
$$
In other words, $\lat{r}{R}{M}$ is the dual lattice to 
$\mt{\wedge^{r} M^*}$ in $\wedge^{r} M \otimes F$.
\end{defn}

\begin{lem}
\label{2.5new}
We have
$
\mt{\wedge^r M} \subset \lat{r}{R}{M},
$
with equality if $|\Gamma| \in \D^\times$ or if $r=1$.
\end{lem}

\begin{proof}
The inclusion follows directly from the definition, and for the rest 
see \cite[Proposition 1.2]{stark}.  (If $|\Gamma| \in \D^\times$ the equality
holds because $\mt{M}$ is a projective $R$-module.)
\end{proof}

\begin{lem}
\label{rs}
If $\bm \in \lat{r}{R}{M}$ and $\bpsi \in \wedge^s M^*$ with $s \le r$, then 
$\bpsi(\bm) \in \lat{r-s}{R}{M}$.
\end{lem}

\begin{proof}
If $\bpsi' \in \wedge^{r-s} M^*$ then 
$\bpsi'(\bpsi(\bm)) = (\bpsi\wedge\bpsi')(\bm) \in R$ 
because $\bm\in\lat{r}{R}{M}$, 
so $\bpsi\wedge\bm \in \lat{r-s}{R}{M}$ by definition.
\end{proof}

\begin{prop}
\label{Rpair}
Suppose $M$ is an $R$-module that is projective as an $\D$-module, and 
$B$ is an $\D$-module. 
For every $s \le r$ and $\rho \in F[\Gamma]$, the construction of Definition \ref{wdg1} 
induces a canonical pairing
$$
(\lat{r}{{R}}{M})[\rho] \times \w{r-s}\Hom_{R}(M,{R} \otimes_\D B) 
   \too (\lat{s}{{R}}{M})[\rho] \otimes_\D B^{\otimes(r-s)}.
$$
In particular, when $s=0$ this pairing takes values in $R[\rho] \otimes_\Z B^{\otimes r}$.
\end{prop}

\begin{proof}
There are natural isomorphisms 
$$
\Hom_{R}(M,{R}) \cong \Hom_\D(M,\D), \quad 
   \Hom_{R}(M,{R} \otimes_\D B) \cong \Hom_\D(M,B).
$$
Since $M$ is a projective $\D$-module, the natural map 
$$
M^* \otimes_\D B \too \Hom_{R}(M,{R} \otimes_\D B)
$$
is an isomorphism.
This isomorphism gives the first map of
\begin{align*}
(\lat{r}{{R}}{M})[\rho] \times \w{r-s}\Hom_{R}(M,&{R} \otimes_\D B) 
   \isom (\lat{r}{{R}}{M})[\rho] \times \w{r-s}(M^* \otimes_\D B_n) \\
   &\too (\lat{r}{{R}}{M})[\rho] \times (\w{r-s}M^*) \otimes_\D B^{\otimes (r-s)} \\
   &\too (\lat{s}{{R}}{M})[\rho] \otimes_\D B^{\otimes (r-s)}
\end{align*}
and the last map comes from Definition \ref{wdg1}, using Lemma \ref{rs}.

If $s=0$, then $\lat{0}{{R}}{M} = {R}$ by definition.
\end{proof}

\begin{rem}
If (for example) 
$m_1,\ldots,m_r \in M$, $\phi_i,\ldots,\phi_r \in \Hom_{R}(M,{R} \otimes_\D B)$, 
and $s = 0$, then the pairing of Proposition \ref{Rpair} is given by 
$$
(\wk{r}m, \wk{r}{\phi}) \mapsto \det(\phi_i(m_j)).
$$
The content of Proposition \ref{Rpair} is that this pairing is defined on all of
$\lat{r}{R}{M}$, not just on $\mt{\wedge^r M }$.
\end{rem}

\end{document}